\documentclass[11pt,reqno]{amsart}

\usepackage[english]{babel}
\usepackage{amsmath}
\usepackage{amsthm}
\usepackage{graphicx}
\usepackage{upgreek}
\usepackage{hyperref}
\usepackage{mathrsfs}
\usepackage{xcolor}
\usepackage{bm,bbm}
\usepackage{amsfonts}
\usepackage{amssymb}
\usepackage{csquotes}
\usepackage{stmaryrd}
\usepackage{ulem}
\usepackage{tikz-cd}
\usepackage{algorithm}
\usepackage{algpseudocode}
\usepackage{enumerate}
\usepackage{soul}

\usepackage[margin=1in]{geometry}

\newtheorem{definition}{Definition}
\newtheorem{remark}{Remark}
\newtheorem{proposition}{Proposition}
\newtheorem{lemma}{Lemma}
\newtheorem{theorem}{Theorem}

\newtheorem{assumption}{Assumption}

\newcommand{\R}{\mathbb{R}}

\newcommand{\N}{\mathbb{N}}

\newcommand{\eps}{\varepsilon}

\newcommand{\al}{\alpha}

\newcommand{\EE}{{\mathbb E}}

\newcommand{\PP}{{\mathbb P}}

\newcommand{\calA}{{\mathcal A}}

\newcommand{\calC}{{\mathcal C}}

\newcommand{\calK}{{\mathcal K}}
\newcommand{\calL}{{\mathcal L}}

\newcommand{\calN}{{\mathcal N}}
\newcommand{\calO}{{\mathcal O}}
\newcommand{\calP}{{\mathcal P}}
\newcommand{\calQ}{{\mathcal Q}}

\newcommand{\calU}{{\mathcal U}}
\newcommand{\calV}{{\mathcal V}}

\newcommand{\calX}{{\mathcal X}}

\newcommand{\scrM}{\mathscr{M}}

\newcommand{\scrE}{\mathscr{E}}
\newcommand{\scrQ}{\mathscr{Q}}

\newcommand{\pl}{\partial}

\newcommand\iy{\infty}

\newcommand{\A}{\mathbb{A}}

\newcommand{\1}{\boldsymbol{\mathbbm{1}}}

\DeclareMathOperator*{\argmin}{arg\,min}
\DeclareMathOperator*{\argmax}{arg\,max}
\newcommand{\CDGME}{C_{DGME}}

\begin{document}

\title{Deep Backward and Galerkin Methods for the Finite State Master Equation}

\author{Asaf Cohen, Mathieu Lauri\`ere, Ethan Zell}\thanks{* This is the final version of the paper. To appear in {\it Journal of Machine Learning Research
}.}

\begin{abstract}
This paper proposes and analyzes two neural network methods to solve the master equation for finite-state mean field games (MFGs). Solving MFGs provides approximate Nash equilibria for stochastic, differential games with finite but large populations of agents. The master equation is a partial differential equation (PDE) whose solution characterizes MFG equilibria for any possible initial distribution. The first method we propose relies on backward induction in a time component while the second method directly tackles the PDE without discretizing time. For both approaches, we prove two types of results: there exist neural networks that make the algorithms' loss functions arbitrarily small and conversely, if the losses are small, then the neural networks are good approximations of the master equation's solution. We conclude the paper with numerical experiments on benchmark problems from the literature up to dimension $15$, and a comparison with solutions computed by a classical method for fixed initial distributions.
\end{abstract}
\maketitle

\tableofcontents


\section{Introduction}

\subsection{Introduction to MFGs}

In 1950, John Nash introduced the concept of a Nash equilibrium, an idea that describes when players trying to minimize an individual cost have no incentive to change their strategy~\cite{nash1951non}. Nash proved that, under mild conditions, a Nash equilibrium always exists in mixed strategies. The question remained however, even when a Nash equilibrium was known to exist, whether there is an algorithmic way to solve for it. It turns out that, in general, computing a Nash equilibrium is a very difficult task~\cite{daskalakis2009complexity}.


The origins of limiting models for many-agent games can be traced back to \cite{Aumann1964} and to  \cite{Schmeidler1973}, both addressing one-shot games. In the realm of stochastic dynamical games, mean field game (MFG) theory was pioneered by \cite{LasryLions2, LasryLions}, along with an independent development by \cite{Huang2006}. MFG paradigms study situations in which a very large number of agents interact in a strategic manner. Specifically, MFGs provide one way to approximate Nash equilibria of large, anonymous, symmetric finite-player games.

%
%
%
%
%
%
%
In characterizing an equilibrium in MFGs, we introduce a {\it representative player} who, instead of reacting to $n$ players, responds to a flow of measures. The {\it MFG equilibrium} is consequently defined by a fixed point, ensuring that the distribution of the dynamics of the representative player under optimality aligns with the flow of measures to which they respond. It is worth noting that Nash equilibria in finite-player games are similarly described using fixed points. 

Returning to the motivating question for $n$-player games, an MFG equilibrium provides an approximate Nash equilibrium in a corresponding finite-player game: if all the $n$ players use the strategy determined by the MFG, each player can be at most $\eps_n$ better off by choosing another strategy, and $\eps_n$ converges to zero as $n\to\infty$. This approximation is valid for games in which players are homogeneous and the interactions are symmetric, in the sense that they occur through the players' empirical distribution. 
While static games are prevalent in game theory for the sake of simplicity, dynamic games are more realistic for applications. In this work, we consider the MFG studied by \cite{bay-coh2019, cec-pel2019, GomesMohrSouza_discrete} which approximates an $n$-player stochastic differential game with a finite number of states. 

Classically, an MFG equilibrium can be described by a coupled forward-backward system of equations, called the {\it MFG system}. The backward equation is a Bellman-type equation describing the evolution of a representative player's {\it value function} (the optimal cost), from which the optimal strategy can be deduced. The forward equation describes the evolution of the population's distribution, which coincides with the distribution of one representative player over the state space. We refer to the books by \cite{Bensoussan2013} 
and \cite{CarmonaDelarue_book_I} for more details on this approach using partial differential equations (PDEs) or stochastic differential equations (SDEs). In finite-state MFGs as we will consider in this work, the forward-backward system is a system of ordinary differential equations (ODEs); the interested reader is referred to~\cite[Chapter 7.2]{CarmonaDelarue_book_I}. 

The solution to the forward-backward system depends on the initial distribution of the population. The representative player's \textit{value function} and control depend only \textit{implicitly} on the population distribution. This dependence can be expressed \textit{explicitly} through the notion of a decoupling field. The MFG equilibrium can be described by the so-called \textit{master equation}, introduced by \cite{Lionsvideo}. The interested reader is referred to the monograph by \cite{CardaliaguetDelarueLasryLions} for a detailed analysis of the master equation and the question of convergence of finite-player Nash equilibria to MFG equilibria. The solution to this equation is a function of the time, the representative player's state, and the population distribution, which will be denoted respectively by $t$, $x$ and $\eta$. The master equation derives from the MFG system (see Proposition~\ref{prop:me_results} below), whose structure is crucial for our analysis.  Besides fully characterizing the MFG equilibrium for any initial distribution, the master equation is also used to study the convergence of the finite player equilibrium to its MFG counterpart \cite{CardaliaguetDelarueLasryLions, del-lac-ram2019,delarue2020master}. In this paper, we are interested in the finite state, finite horizon master equation studied by \cite{bay-coh2019, cec-pel2019}:
\begin{equation}\label{master_equation:finite_horizon}
    \begin{cases}
    \pl_t U(t,x,\eta) + \sum_{y,z\in [d]} \eta_y D^\eta_{yz} U(t,x,\eta) \gamma_z^*(y,\Delta_y U(t,\cdot,\eta)) + \bar H(x,\eta,\Delta_x U(t,\cdot,\eta)) = 0, \\
    U(T,x,\eta) = g(x,\eta), \qquad (t,x,\eta) \in [0,T)\times [d]\times \calP([d]).
    \end{cases}
\end{equation} 
Above, $[d]:=\{1,\dots, d\}$ is the state space, $T$ is the terminal time of the game, $\bar H$ is the Hamiltonian (defined in Section~\ref{sec:notation}), and $g$ is a terminal condition. The function $\gamma^*$ denotes the minimizer of the Hamiltonian, which under regularity assumptions is equal to the gradient of the Hamiltonian. The variables are the time $t$, the state $x$, and the distribution $\eta$, which lies in the $(d-1)$-dimensional simplex $\calP([d])$. 
We use the notation  $\Delta_x U(t,\cdot,\eta):= (U(t,y,\eta)-U(t,x,\eta))_{y\in [d]}$ for each $x\in [d]$. Derivatives on the simplex, denoted by $D^\eta_{yz} U(t,x,\eta)$, can be thought of as usual directional derivatives and will be explicitly defined in Section~\ref{sec:notation}. 

Note that \eqref{master_equation:finite_horizon} is, in general, a nonlinear PDE, for which there is no analytical solution and the question of well-posedness is challenging (for an overview, see the book by \cite{MR2597943}). In fact, even computing numerically an approximate solution is a daunting problem. When $d\le 3$, traditional methods such as finite difference schemes can be employed. However, the computational cost of these methods becomes prohibitive for dimensions greater than $3$. For example, in grid-based methods, the number of points grows exponentially quickly in $d$. Other methods include Markov chain approximations, see \cite{BBC2017b}. We refer to \cite{MR4378818} for an overview of classical methods for PDEs and~\cite{achdoumean,MR4368188} for an overview of classical methods for MFGs.  Besides, many applications require $d > 3$. And additionally, \cite{Ber-cec2012} recently demonstrated that MFGs with a continuous state space can be approximated by MFGs with finite, but large, state spaces. This approximation relies on the convergence of the solution of the associated master equations as the number of states tends to infinity. Neural networks provide one avenue to mitigate the curse of dimensionality, and this has triggered interest for machine learning-based methods, as we discuss in the next section. 


The value function under the unique Nash equilibrium in the $n$-player game converges to the solution to the master equation up to an error of order $\calO(n^{-1})$ (and so do the optimal controls) \cite{bay-coh2019, cec-pel2019}. Then, we prove that the two algorithms we propose, the deep backward master equation (DBME) and deep Galerkin master equation (DGME) methods, approximate the master equation which, as discussed in the prior paragraph, is not easy to solve analytically. So while the primary object of interest for this paper is the master equation, our main contribution completes the picture for how to provably approximate the solution to the $n$-player game.

\subsection{Overview of Machine Learning for PDEs}


In this work, we propose and rigorously analyze deep learning methods to tackle the master equation~\eqref{master_equation:finite_horizon} for finite-state MFGs beyond what is possible with traditional numerical methods. Our numerical methods build upon the recent development of deep learning approaches for PDEs. In the past few years, numerous methods have been proposed along these lines, such as the deep BSDE method \cite{MR3736669}, the deep Galerkin method \cite{MR3874585}, physics informed neural networks~\cite{raissi2019physics} and deep backward dynamic programming~\cite{MR4322044}. In the context of MFGs and mean field control problems, deep learning methods have been applied using the probabilistic approach~\cite{carmona2022convergence,fouque2020deep,germain2022numerical} or the analytical approach~\cite{al2022extensions, cao2020connecting,carmona2021convergence,ruthotto2020machine}. The aforementioned works consider neural networks whose inputs are only the individual state and possibly time. A few works have considered neural networks taking as inputs a representation of the mean field. The paper by \cite{germain2022deepsets} proposed a deep backward method for mean field control problems. As with the first method (called DBME) we present below, it is based on a discrete-time, backward scheme. The rate of convergence in terms of number of particles was studied in~\cite{germain2022rate}. However, this method solves the Bellman equation for mean field control problems while we solve the master equation for MFGs, which fails to satisfy a comparison principle. Furthermore, \cite{germain2022deepsets} considered a continuous space model using particle-based approximations while we focus on a finite space model. Work by \cite{dayanikli2023deeppop} solves continuous space mean field control problems using population-dependent controls by directly learning the control function instead of using a backward dynamic programming equation. As for MFG with population-dependent controls, \cite{perrin2022generalization} proposed a deep reinforcement learning method for finite-state MFGs based on a fictitious play and deep Q-networks. In contrast, our two methods do not rely on reinforcement learning and we provide a detailed analysis of the algorithms' errors. We refer to the overview of \cite{hu2023recent} for a review of deep learning methods in optimal control and games.

In this work, we will focus on two types of methods. The first approach finds its roots in the pioneering work of ~\cite{MR3736669}, which exploits the connection between SDEs and PDEs: using Feynman--Kac type formulas, the PDE solution satisfies a backward SDE, that is, they satisfy a terminal condition. Corresponding to the SDE, the PDE has a terminal condition. Such SDEs and PDEs crop up often in mathematical finance, most iconically in the Black--Scholes equation for pricing European options \cite{MR3363443}. 
While the original method of~\cite{MR3736669} consists in replacing the BSDE by a shooting method for a forward SDE, subsequent works exploit the backward structure differently. In particular, deep backward methods were explored in control problems in the nonlinear and fully-nonlinear cases \cite{MR4081911, MR4322044}. The Deep Backward Dynamic Programming algorithm (DBDP) from \cite{MR4081911} (more specifically their DBDP2 algorithm) uses a sequence of neural networks on a discretization of $[0,T]$. Training is done by starting from the terminal time and, going backward in time, the neural network for a given time step is trained to minimize a loss function which involves the neural network at the following time step.


The second method we propose is the DGME, introduced as the deep Galerkin method by \cite{MR3874585}. The PDE solution is approximated by a neural network, which is a function of the time and space variables. Training is done by minimizing an empirical loss which captures the residual of the PDE, as well as possible boundary conditions. The loss is computed over points sampled inside the domain and, if needed, on its boundary. In this way, the method does not require any discretization of space or time. The solution is learned over the whole domain thanks to the generalization capability of neural networks. 

For these methods, most theoretical results rely on universal approximation theorems for suitable classes of functions related to the regularity of PDE solutions. In that sense, other classes of approximators could be used instead of neural networks. However, neural networks seem to be a very suitable choice to solve PDEs. Two of the main limitations of neural networks are the lack of explainability and the large amount of data required for training, but these are not issues when solving PDEs with the aforementioned methods. Indeed, explainability is not a priority when approximating PDE solutions, and data points are obtained by sampling over the PDE domain as many times as desired.

\subsection{Contributions and Challenges}  

While the well-posedness of \eqref{master_equation:finite_horizon} was studied by \cite{bay-coh2019} and \cite{cec-pel2019}, no analytical solution is known. Therefore we provide two algorithms, we prove their correctness, and numerically solve the master equation \eqref{master_equation:finite_horizon} for some examples.


\subsubsection{DBME}

The first method we propose, called deep backward master equation (DBME), is inspired by the DBDP of~\cite{MR4081911}, although there are key differences. An important disparity lies in the corresponding stochastic model for each setting: the diffusion dynamics found in the underlying DBDP problem are substituted with jump dynamics in the finite state MFG. Essentially, in the DBDP the target function is twice differentiable; this allows the use of correspondence between certain martingale and loss terms that would otherwise be hard to bound. Unfortunately in the case of the master equation, it is not necessarily true that the solution $U$ is twice differentiable. Another crucial distinction lies in the fact that the DBDP is designed to fit Hamilton--Jacobi--Bellman equations originating from optimal control problems. In contrast, the MFG solution manifests as a fixed point of a best response mapping within a control problem and consequently, the master equation lacks a so-called comparison principle, which leads us to design distinct algorithms and proofs of convergence. To establish the convergence, we leverage the transformation of the master equation into the coupled, forward-backward MFG system for a fixed initial distribution. At a technical level, another difference lies in the fact that the DBME minimizes a maximum-based error, while the DBDP minimizes an $L^2$-error.  All these differences necessitate a different algorithm and a different analysis. To overcome these difficulties, the analysis for the DBME, like that of the DGME, makes use of the MFG system directly. Our approach contrasts with that of the DBDP which uses propagation by SDEs.

\subsubsection{DGME}

The second algorithm we consider builds upon the DGM of~\cite{MR3874585} and adapts it to solve the master equation~\eqref{master_equation:finite_horizon}, as explored by \cite[Algorithm~7]{MR4368188}. That DGM adaptation is recalled here in Algorithm \ref{alg:DGM_ME}, with a modification. As used by \cite{MR4368188}, the DGM minimizes the expectation of an $L^2$-error. In our formulation, the DGME minimizes a worst-case loss function formulated as a maximum that is, in practice, sampled and not a true maximum over the entire state space. While {\it applying} the DGM to the master equation is not novel, the finite state master equation is a different form of PDE than that studied by \cite{MR3874585} and so the convergence theorem (Theorem 7.3) of their work does not apply. Moreover, no proofs of DGM convergence were offered by \cite{MR4368188}.

We prove the convergence of our DGME algorithm to the unique master equation solution in Theorem \ref{thm:dgm_convergence}. Our result is analogous to that of \cite[Theorem~7.3]{MR3874585}, but as previously mentioned, since the PDE structure differs, our proof requires entirely different methods. Namely, recall that the master equation can be constructed from the MFG system and our proof relies on the structure of the MFG system. Using the neural network trained by the DGME, we construct an approximate MFG system and use MFG techniques (specifically, MFG duality) to compare the approximate solution and the true solution. 
As a consequence of our modification and the structure of the master equation, we obtain convergence of the neural network approximation to the true solution in the supremum norm, which is different from and arguably stronger than the $L^2$-convergence obtained by \cite{MR3874585}. 
Broadly speaking, this choice allows us to deal with the maximum over all states instead of an expectation and hence our analysis is state-agnostic. One more subtle difference is that \cite{MR3874585} tackle the case of a neural network with one-hidden layer for simplicity's sake, while the proofs we offer for the DGME apply to feedforward networks of arbitrary depth. This generality comes from our use of the universal approximation theorem \cite[Theorem~3.1]{Hornick90}.

\subsubsection{The convergence results} For each algorithm, we have two main types of results. The first type of result, found in Theorems \ref{thm:dbme_existence} and \ref{prop:existence}, builds on universal approximation properties of neural networks. We show that there exist neural networks that approximate the master equation solution and that when this is the case, the algorithms' loss functions are small. 

The second type of result, found in Theorems \ref{thm:DBME} and \ref{thm:dgm_convergence}, asserts that when the algorithm results in a neural network (or in the case of the DBME, a family of neural networks) with small loss, then the obtained network is in fact close to the true solution in supremum norm. The rate of convergence in each case depends on the loss value and, in the case of the DBME only, the size of the time partition. Note that the empirical loss achievable in practice by the algorithm depends on the depth and width of the neural network, the number of epochs it is trained for, the amount of training data, and so on. This type of bound is morally different from that of \cite{MR4081911}, which uses a constant obtained from universal approximation.

\subsection{Structure of the Paper}

We begin by defining notation in Section~\ref{sec:notation} before providing a description of the MFG model and recalling the major results concerning the MFG and the master equation in Section~\ref{sec:me_recall}. In Section \ref{sec:nns} we outline the architecture of the neural networks we consider. Then in Sections~\ref{sec:dbme} and~\ref{sec:dgm}, we describe and analyze the DBME and DGME, respectively. Finally, in Section~\ref{sec:numerics}, we present numerical results for two examples using these algorithms.

\section{Notation}\label{sec:notation}

We let $[d]:=\{1,\dots,d\}$ be the finite set of states. The finite difference notation for a vector $b\in\R^d$ is $\Delta_x b:= (b_y-b_x)_{y\in [d]}$, $x\in [d]$. Generally, $x,y,z\in [d]$ unless otherwise stated. A probability measure over $[d]$ is identified with the $(d-1)$-dimensional simplex in $\R^d$, which we denote $\calP([d])$; by this we mean that:
\[
    \calP([d]) := \Big\{\eta\in \R^d : \sum_{x\in [d]} \eta_x = 1, \quad \eta_x \geq 0 \Big\}.
\] Unless otherwise stated, $\eta,\mu \in \calP([d])$. For a vector in $\R^k$, $k\in\N$, $|\cdot|$ denotes the usual Euclidean norm. When $A$ is a finite set, $|A|$ denotes the cardinality of the set. For any Euclidean spaces $\scrE_1,\dots,\scrE_{k_1}$ and any measurable function $\varphi:\prod_{j=1}^{k_1} \scrE_j \to \R^{k_2}$, $k_1,k_2\in\N$, we denote by $\|\varphi\|_\iy$ the supremum norm of $\varphi$. In the case that some arguments of $\varphi$ are specified, the essential supremum is taken over only the unspecified arguments. To alleviate the notations, we will denote by $C_\varphi$ the supremum norm of a bounded function $\varphi$, and by $C_{L,\varphi}$ the Lipschitz constant of a Lipschitz function $\varphi$.

The measure derivative on the simplex $\calP([d])$ is written $D^\eta$ (where the direction of the derivative is specified below). 
Since the simplex is finite-dimensional, one can think of $D^\eta$ as the usual directional derivative; that is, for each $\eta\in\calP([d])$ and for each $y\in[d]$, such that $\eta_y>0$, we define:
\begin{equation}
    \label{eq:def-D-eta}
    D^\eta_{yz} \varphi(\eta) := \lim_{h\searrow 0} \frac{\varphi(\eta+e_{yz}h) - \varphi(\eta)}{h},
\end{equation}
where $e_{yz}$ is the $z$-th standard basis vector of $\R^d$ minus the $y$-th, and $\varphi:\calP([d]) \to \R$ is a measurable function such that the limit exists. Also, $D^\eta_{y} \varphi := (D^\eta_{yz} \varphi)_{z\in [d]}$ is a vector. Note that by definition $D^\eta_{yy} \varphi(\eta)=0$. Furthermore, in equation \eqref{master_equation:finite_horizon} this derivative is multiplied by $\eta_y$, so its definition at $\eta_y=0$ does not affect the expression. Additionally, $D^\eta_{yy}$ in \eqref{master_equation:finite_horizon} is multiplied by $\gamma_y^*(y,\Delta_y U(t,\cdot,\eta))$, which does not represent a rate of transition. Alternatively, we could define the derivatives for $y \ne z$ only and use the double sum in \eqref{master_equation:finite_horizon} as $\sum_{y\in[d]}\sum_{z\in[d], z\ne y}$.

Let $\calC^{1,1}_{}([0,T] \times \calP([d]))$ be the set of $\R$-valued functions defined on $[0,T] \times \calP([d])$, where the time derivative is continuous and whose measure derivative $D^\eta$ is Lipschitz. The space $\calC^{1,1}([0,T] \times \calP([d]))$ can be thought of as combining the Banach space of continuously differentiable, real-valued functions $\calC^1([0,T])$ with the H\"older space $\calC^{1,1}(\calP([d]))$ as defined in \cite[Section~5.1]{MR2597943}. Let $\scrE$ be a Euclidean space and, for a given differentiable function $R:\scrE \to\R$ with derivative $D^{\scrE} R$, we define the $(1, 1)$-H\"older seminorm as: 
\[
    [R]_{\calC^{1,1}} := \sup_{p\neq q} \Big\{\tfrac{|D^{\scrE} R(p) - D^{\scrE} R(q)|}{|p-q|}\Big\}.
\] The first ``$1$" in the notation $\calC^{1,1}$ refers to the order of the derivative, while the second refers to the modulus of continuity. Then, the $(1,1)$-H\"older norm is:
\[
    \|R\|_{\calC^{1,1}} := [R]_{\calC^{1,1}} + \| D^{\scrE} R \|_{\iy} + \|R\|_\iy. 
\] 
With this in mind, we endow the space $\calC^{1,1}([0,T]\times\calP([d]))$ with the norm:
\begin{align}\label{eqn:main_norm}
    \|g\|_{\calC^{1,1}}:=  \|g\|_\iy +\|\pl_t g\|_\iy +  \|D^\eta_1 g\|_\iy + \sup_{t\in[0,T]}[g(t,\cdot)]_{\calC^{1,1}},
    \end{align} 
where $D^\eta_{yz}$ is defined in~\eqref{eq:def-D-eta}, and recall that, by definition of $\|\cdot\|_\iy$, the supremum is taken over all $(t,\eta)\in\calP([d])\times[0,T]$. Note that $\calC^{1,1}([0,T]\times \calP([d]))$ is a Banach space with $\|\cdot\|_{\calC^{1,1}}$. 

A few times in the paper we will mention the space of continuously differentiable functions $\calC^1(X, Y)$ where $X$ is either $[t_0,T]$ or $[t_0,T] \times [d]$ with $t_0 \in [0,T]$ and where $Y$ is some subset of euclidean space. We endow this space with the norm:
\[
\|g\|_{\calC^1(X, Y)} := \|g\|_\iy + \|\partial_t g\|_\iy,
\] making it Banach.

\section{The Finite State Master Equation and the MFG}\label{sec:me_recall}

In this section, we continue with structure and the notation adopted by \cite{bay-coh2019, cec-pel2019} with only minor additions. For readability, we recall the notation here. Throughout we fix $(\Omega, \mathscr{F},(\mathscr{F}_t)_t, \PP)$, a filtered probability space in the background satisfying the usual conditions.

\subsection{The MFG} We consider a representative player that uses Markovian controls taking values in a set of rates $\A\subseteq \R_+ := [0,+\iy)$. We will consider rates bounded in a compact set; that is, $\A := [\mathfrak{a}_l, \mathfrak{a}_u]$, with $0\leq \mathfrak{a}_l \leq \mathfrak{a}_u < +\iy$. A {\it Markovian control} is  a measurable function $\alpha:\R_+\times [d]\to \A^d_{[d]}$ where $\A^d_{[d]}=\cup_{x\in[d]}\A^d_{-x}$, and:
\[
\A_{-x}^d := \Big\{ a\in \R^d \mid \forall y\neq x, \quad a_y \in\A, \quad a_x = -\sum_{y\neq x} a_y\Big\}.
\] Here, for any $y\neq x$, $\al_y(t,x) := \al(t,x)_y \in \A$ is the rate of transition at time $t$ to move from state $x$ to state $y$ and as usual with continuous-time Markov chains, $\alpha_x(t,x)=-\sum_{y,y\ne x}\alpha_y(t,x)$. We denote by $\calQ[\A]$ the set of $d\times d$ transition-rate matrices with rates in $\A$. We denote by $\calA$ the set of measurable mappings $\alpha:[0,T]\to\calQ[\A]$. 

Fix $t_0 \in [0,T]$ to be the initial time. The jump dynamics of the representative player's state $\calX$ are given by:
\begin{equation}\label{eqn:jump_dynamics}
    \calX_t = \calX_{t_0} + \int_{t_0}^t \int_{\A^{d}} \sum_{y\in [d]} (y-\calX_{s^-})\boldsymbol{1}_{\{\xi_y\in(0,\al_y(s,\calX_{s^-}))\}} \calN(ds,d\xi), \qquad t\in[t_0,T].
\end{equation} 
where $\calN$ is a Poisson random measure with intensity measure $\nu$ given by
\begin{equation}\label{n_intensity_measure}
    \nu(E) := \sum_{y\in [d]}\text{Leb}(E\cap \A^{d}_y),
\end{equation} 
with $\A^{d}_y := \{a\in\A^d \mid a_x = 0 \text{ for all }x\neq y\}$, where $\text{Leb}$ is the Lebesgue measure on $\R$, and $E$ is any Borel measurable set. We can think of 
$\A^d_y$ as the $y$-projection of 
$\A^d$, the set of rates for each state transition; thus, $\A^d_y$ is the set of transition rates from state $y \in [d]$.


In order to set up the MFG equilibrium, we describe the cost structure of the game with respect to a representative player, who selects a Markovian control $\al$ in order to  play against a smooth flow of measures $\mu \in\calC^{1}([t_0,T], \calP([d]))$. 
We emphasize that the representative player's cost includes the flow of measures $\mu$. In this work, we make use of the regularity result from \cite{cec-pel2019}, which assumes that the {\it total running cost} is separable in the form $f + F$. The running cost is $f:[d]\times\A^d_{[d]}\to\R$, a function of the current state and control, while the mean field cost is $F:[d]\times \calP([d]) \to\R$, a function of the state and mean field distribution. Recall that $\alpha_x(t,x)$ does not represent a rate, 
thus $f(x,a)$ is assumed independent of $a_x$ as a technicality.
Given a flow $\mu \in\calC^{1}([t_0,T], \calP([d]))$ of distributions, the representative agent's control problem is to minimize: 
\begin{equation}
\label{eq:def-J-integral}
    J(t_0,x,(\mu(s))_{s\in [t_0,T]},\al) = \EE_{(x,\eta)} \Big[g(\calX_T, \mu(T)) + \int_{t_0}^T \big[f(\calX_s, \al(s,\calX_s)) + F(\calX_s, \mu(s)) \big]ds\Big],
\end{equation} 
over all $\al\in\calA$ where $(\calX_{t_0},\mu(t_0)) = (x,\eta) \in [d]\times \calP([d])$.
Given the flow of measure $\mu$ and an initial state $\calX_{t_0} = x \in[d]$, the distribution (or law) of the representative player, denoted $(\calL(\calX_t))_{t\in [t_0,T]}$, is a fully deterministic flow of measures on $\calP([d])$. Formally, when $(\calL(\calX_t))_{t\in[t_0,T]}=(\mu(t))_{t\in[t_0,T]}$  and $\al$ minimizes the cost:
\[
    J(t_0, x, (\mu(t))_{t\in [t_0,T]}, \al),
\] 
we say that we have a \textit{mean field game equilibrium} and we refer to this optimal cost as the \textit{value function}. To be mathematically precise, the MFG equilibrium is a fixed point, as described next.
\begin{definition}[MFG equilibrium]\label{def:mfe}
    Fix $t_0\in [0,T]$ and $\mu_0 \in \calP([d])$. A pair $(\hat\al,\hat\mu) \in \calA \times \calC^{1}([t_0,T], \calP([d]))$  is a MFG equilibrium over the interval $[t_0,T]$ for the initial distribution $\mu_0$ if the following two conditions are satisfied. First, $\hat\al$ is an optimal control for $\EE_{x \sim \mu_0}[J(t_0,x,\hat\mu,\hat\al)]$ and second, for every $t \in [t_0,T]$, $\hat\mu(t) = \calL(\calX^{\hat\al}_t)$, where $\calX^{\hat\al}$ solves~\eqref{eqn:jump_dynamics} controlled by $\hat\al$.
\end{definition}

Crucially, notice that the equilibrium is defined for a fixed initial distribution $\mu_0$. For different initial distributions, we generally obtain different equilibrium controls and mean field flows. The master equation characterizes MFG equilibria for \textit{any} initial distribution.

As explained below in Section~\ref{sec:mfg-known-results}, the equilibrium can be characterized using optimality conditions which involve the Hamiltonian, defined as: 
\[
\bar H (x,\eta, b) := \min_{a\in\A^d_{-x}} \Big\{f(x,a) + F(x,\eta) + \sum_{y\neq x} a_yb_y \Big\} 
    = H (x,b) + F(x,\eta),
\] where,
\[
H (x,b) := \min_{a\in\A^d_{-x}} \Big\{f(x,a) + \sum_{y\neq x} a_yb_y \Big\}. 
\] At times, using $\bar H$ rather than $H+F$ will simplify the presentation of the PDE. 

\begin{remark}
The separability of the cost and the Hamiltonian is a standard assumption in the MFG literature. This, combined with the monotonicity assumption on functions $F$ and $G$, leads to the uniqueness of the MFG solution and the corresponding master equation solution, see e.g., \cite{CardaliaguetDelarueLasryLions, bay-coh2019, cec-pel2019} as well as \cite[P. 652]{CarmonaDelarue_book_I}.
It is important to note that while these are the prevalent assumptions, there are alternative conditions in the literature, such as alternative monotone conditions, see e.g., \cite{gan-mes2022, gan-mes-mou-zha2022, gra-mes2023}, that imply the uniqueness of solutions to MFG systems and the well-posedness of the corresponding master equations; or the anti-monotonicity condition, which can also establish global well-posedness for mean field game master equations with nonseparable Hamiltonians, see \cite{2022arXiv220110762M}.

We emphasize that the goal of this paper is to leverage the established results from previous works to solve the master equation, rather than to propose new conditions for achieving uniqueness in the solution of the master equation.
\end{remark}

\subsection{Assumptions}

In this section, we list the assumptions made on our MFG model. Such assumptions are standard and appear in previous work by \cite{bay-coh2019, cec-pel2019, CZ2022}. 

Recall that the action set is $\A:=[\mathfrak{a}_l, \mathfrak{a}_u]$ where $0\leq \mathfrak{a}_l \leq \mathfrak{a}_u< \iy$.\footnote{In the papers by \cite{bay-coh2019, CZ2022} it is assumed that $\mathfrak{a}_l >0$; however in his thesis, Cecchin managed to relax the assumption to allow for $\mathfrak{a}_l =0$ \cite{cec-pel2019}.} 

\begin{assumption}
    The Hamiltonian has a unique minimizer to which we refer to as the {\rm optimal rate selector} and is denoted:
\begin{equation}
\label{eq:def-gammastar}
    \gamma^* (x,b) := \argmin_{a\in\A^d_{-x}} \big\{f(x,a) + a\cdot b\big\}.
\end{equation} The optimal rate selector $\gamma^*$ is a measurable function. 
\end{assumption}

We note that since $\gamma^*$ is given by~\eqref{eq:def-gammastar}, one sufficient condition to guarantee this is when $f$ is strictly convex in $a$.

As for the mean field cost $F$, and the terminal cost $g$, we assume 
\begin{assumption}
    The functions $F$ and $g$ are continuously differentiable in $\eta$ with Lipschitz derivative; namely, for any $x\in [d]$ $F(x,\cdot) \in \calC^1(\calP([d]))$ and $D^\eta_{1z} F(x,\cdot)$ is Lipschitz. Moreover, $F$ and $g$ are Lasry--Lions monotone, meaning: for both $\phi=F,g$,
    \begin{align}\label{eqn:LL-monotone}
        \sum_{x\in [d]} (\phi(x,\eta) - \phi(x,\hat\eta))(\eta_x - \hat\eta_x) \geq 0,\qquad\forall \eta, \hat\eta\in\calP([d]).
    \end{align} 
\end{assumption} 
Intuitively, the Lasry--Lions monotonicity condition means that a representative player would prefer to avoid congestion to decrease their cost. 

In the following assumption, we specify the regularity of $H$, which, since it is defined through $f$ implicitly defines most of the necessary regularity for $f$. Moreover, since $F$ is defined on a compact set, its regularity assumptions imply it is bounded.

\begin{assumption}
    Assume $f$ is bounded. We let 
    \begin{align}\label{eqn:interval_restriction}
        W := \sqrt{2d}[T(C_f + C_F) + C_g] + 1,
    \end{align} 
    and we assume that, on the compact set $[-W, W]$, 
    the derivatives $D^2_{pp}H$ and $D_p H$ exist and are Lipschitz in $p$. Moreover, there exists a positive constant $C_{2,H}$ such that:
    \begin{align}\label{eqn:H_regularity}
        D^2_{pp} H(x,p) \leq -C_{2,H}.
    \end{align} 
\end{assumption}

Note that when $H$ is differentiable, \cite[Proposition~1]{Gomes2013} proved:
\begin{align}\label{eqn:DH}
    \gamma^*(x,p) = D_p H(x,p).
\end{align}

In order to keep this level of generality for the regularity of $H$ (namely, that these additional assumptions hold on $[-W,W]$ but not necessarily globally), we will on occasion prove that a particular $p$-argument of $H$ is uniformly bounded by $W$. For the purpose of this paper, the assumption on $W$ in \eqref{eqn:interval_restriction} is not too restrictive. Computations with the Hamiltonian only appear with the argument $\Delta_x U$ or $\Delta_x \calU$, where $\calU$ is a neural network specified in the following sections that approximates $U$. As we explain in detail in Remark \ref{rmk:nn_bound}, both of these arguments are bounded by $W$ from \eqref{eqn:interval_restriction} and hence the Hamiltonian will always be regular enough when it is required.

\begin{remark}
\label{rem:gammastar-Lip}
    From the above assumption on $H$, we have that $\gamma^*$ is locally Lipschitz. Note that a sufficient condition is that $f$ is uniformly convex in $a$. 
\end{remark}

Throughout the paper $C$ denotes a generic, positive constant that depends only on the problem's parameters (that is, the parameters introduced above). To alleviate excessive notation, $C$'s value may change from one line to another.

\subsection{Known Results}\label{sec:mfg-known-results}

Next up, we recall some established results from the study of finite-state, finite-horizon MFGs and master equations. 

The functions $u^{t_0, \eta}$ and $\mu^{t_0, \eta}$ are defined in a moment but we first describe their meanings. The measure $\mu^{t_0, \eta}$ is the evolving MFG equilibrium as in Definition \ref{def:mfe}. The value function $u^{t_0,\eta}$ is the value function (optimal cost) of the MFG starting at time $t_0$ with initial distribution $\eta \in\calP([d])$ along the MFG equilibrium. So, the value $u^{t_,\eta} (t,x)$ is the remaining optimal cost to a player in state $x$ at time $t\geq t_0$ until the game ends, at time $T$. 


The MFG system is:
\begin{align}
    \begin{split}\label{eqn:mfg}
    &\frac{d}{dt} u^{t_0,\eta} (t,x) + \bar H(x, \mu^{t_0,\eta}(t), \Delta_x u^{t_0,\eta} (t,\cdot)) = 0, \qquad (t,x) \in [t_0,T] \times [d],\\ 
    &\frac{d}{dt} \mu^{t_0,\eta} (t,x) = \sum_{y\in [d]} \mu^{t_0,\eta} (t,y) \gamma^*_x(y,\Delta_y u^{t_0,\eta} (t,\cdot)), \qquad (t,x) \in [t_0,T] \times [d] ,\\ 
    &\mu^{t_0,\eta} (t_0,x) = \eta(x), \qquad x \in [d], \\ 
    &u^{t_0,\eta}(T,x) = g(x,\mu^{t_0,\eta} (T)), \qquad x \in [d].
\end{split}
\end{align} Note that the system is composed of a forward equation and a backward equation. 
The above ODE system characterizes the MFG equilibrium for a given initial condition $(t_0,\eta)$. However, our goal is to compute the MFG equilibrium for every initial condition. This would require solving a continuum of ODE systems, which is not feasible. Instead, we will focus on the master equation. Informally, it encompasses the continuum of ODE systems and its solution captures the dependence of the value function on the mean field. 
More precisely, the master equation's solution $U$ is defined through its MFG system. The following proposition is a combination of the results from \cite[Proposition~1, Proposition~5, Theorem~6]{cec-pel2019}\footnote{In fact, we reuse the argument from Proposition~1, but applied to the MFG system.} and \cite[Section 1.2.4]{CardaliaguetDelarueLasryLions}.

\begin{proposition}\label{prop:me_results}
There exists a unique solution, denoted by $(u^{t_0,\eta}, \mu^{t_0,\eta})$, in $\calC^{1}([t_0, T] \times [d], \R) \times \calC^{1}([t_0,T] \times [d], \calP([d]))$ to \eqref{eqn:mfg}. Let $U$ be defined by: 
\begin{align}\label{eqn:U_def}
    U(t_0,x,\eta) := u^{t_0,\eta} (t_0,x).
\end{align} 
Then $U$ is the unique classical solution to the master equation~\eqref{master_equation:finite_horizon}. Moreover, the consistency relation holds: for all $t_0 \in [0,T]$, 
\begin{align}\label{eqn:consistency_relation}
    U(t,x,\mu^{t_0,\eta}(t))=u^{t,\mu^{t_0,\eta}(t)}(t)=u^{t_0,\eta}(t), \qquad (t, x, \eta) \in [t_0,T] \times [d] \times \calP([d]).
\end{align}
Additionally, $U(\cdot, x,\cdot) \in \calC^{1,1}_{}([0,T] \times \calP([d]))$ for every $x\in [d]$. Moreover,
\begin{align}\label{eqn:me_cost_equivalence}
\begin{split}
    U(t_0,x,\eta) 
    &= J(t_0,x,(\mu^{t_0,\eta}(s))_{s\in [t_0,T]},(\gamma^*( \cdot, \Delta_\cdot U(s,\cdot,\mu^{t_0,\eta}(s))))_{s\in [t_0,T]}) 
    \\
    &= \EE_{(x,\eta)}\Big[ g(\calX_T, \mu^{t_0,\eta}(T)) \\
    &\qquad\qquad + \int_{t_0}^T \big[f(\calX_s, \gamma^*(\calX_s, \Delta_{\calX_s} (U(s,\cdot,\mu^{t_0,\eta}(s))))) + F(\calX_s, \mu^{t_0,\eta}(s))\big] ds\Big], 
\end{split}
\end{align}
where the second line is by definition of $J$, see~\eqref{eq:def-J-integral}. 
As a consequence of \eqref{eqn:me_cost_equivalence},
\begin{align}\label{eqn:U_bound}
    |U(t_0,x,\eta)| \leq C_g + T(C_f+C_F)=: \tilde T.
\end{align}
Recalling the definition of $W$ in~\eqref{eqn:interval_restriction}, note that $W = \sqrt{2d} \tilde{T} + 1$.

\end{proposition}

\begin{remark}
    By \eqref{eqn:me_cost_equivalence}, we can think of $U(t_0,x,\eta)$ as the value of the MFG equilibrium when the initial time is $t_0$, the initial state is $x$, and the initial distribution is $\eta$. Moreover, $(\gamma^*_y(x, \Delta_{x} (U(s,\cdot,\mu^{t_0,\eta}(s))))_{x,y\in[d]}$ is the rate matrix under the mean field equilibrium; namely, these are the rates that dictate how quickly the representative player transitions from one state to another under this equilibrium. By \eqref{eqn:consistency_relation}, we have $\gamma^*_y(x, \Delta_{x} (U(s,\cdot,\mu^{t_0,\eta}(s)))) = \gamma^*_y(x,\Delta_x u^{t_0,\eta}(s,\cdot))$, which are precisely the transition rates in the Kolmogorov equation for $\mu^{t_0,\eta} (t)$ from \eqref{eqn:mfg}.
\end{remark}

\begin{remark}
    Recall that the motivation for investigating the master equation in the first place was to solve the corresponding $n$-player game. Per \cite[Theorem~2.1]{bay-coh2019} or \cite[Theorem~1]{cec-pel2019}, the $n$-player game's value function (cost under the unique Nash equilibrium) converges to the master equation's solution with rate $n^{-1}$. Moreover by \cite[Theorem~2, Theorem~4]{cec-pel2019}, the $n$-player Nash equilibrium results in an empirical distribution of players whose trajectory converges to the MFG equilibrium, with rate depending on the notion of convergence. 
\end{remark}

\section{Neural Networks}\label{sec:nns}

With the major points from the MFG literature in hand, we can switch to the neural network side. We define the neural networks to be used as function approximators and outline their parameters. From a proof perspective, the DBME and DGME algorithms require networks with slightly different levels of regularity, so here we only describe what features the networks used for each algorithm have in common. 

Throughout the paper, the use of neural networks is justified by the Universal Approximation Theorem, specifically the version from \cite[Theorem~3.1]{Hornick90}. This version of the Universal Approximation Theorem asserts that any differentiable function defined on a compact domain in Euclidean space is arbitrarily well-approximated (and has derivative arbitrarily well-approximated) by a neural network (and its derivative). This is relevant in our case since we are concerned with the master equation solution, $U: [0,T] \times [d] \times \calP([d]) \to \R$, and we may identify $\calP([d])$ with a compact subset of $\R^{d-1}$.

We will denote the total number of dense hidden layers of a deep, feed-forward, fully-connected neural network as $L$, with the number of parameters in layer $\ell$ as $\delta_\ell$  for all $\ell \leq L$. By convention, the first layer will always be the input layer. For any such neural network, define $\bar\delta$ as the number of parameters. Throughout, the parameters are represented by a trainable vector $\theta \in \R^{\bar\delta}$. 

Depending on the algorithm at hand, the input dimension of the neural networks will differ. For the DBME method, the input dimension is $d+1$ since the time component is discretized as part of the algorithm. In the DGME, the input vector is $(t,x,\eta) \in [0,T] \times [d] \times \calP([d])$, which is of dimension $d+2$; so $\delta_1$ is either $d+4$ or $d+3$ when accounting for the affine transformation parameters in the neural network. We will give a concrete example later in \eqref{eqn:nn_structure}. In either case, the output dimension is $1$.

For simplicity, the activation function will be the same for all layers. We will denote it by $\phi$. So that necessary theoretical results are accessible, $\phi$ must be smooth and nonlinear so the neural networks we consider are smooth universal approximators; for example, the hyperbolic tangent function or the sigmoid function are good choices. We use $\calN\calN$ to denote the set of all such neural networks using activation function $\phi$. We denote the subset:
\begin{align}\label{eqn:NN_cap_def}
    \calN\calN(C_0, C_1) := \big\{ \calU \in \calN\calN : |\calU| \leq C_0 , |D^\eta \calU| \leq C_1 \big\}. 
\end{align} 
Note that what we mean by a neural network is not simply an architecture but an architecture with fixed values for the parameters. The constraints in~\eqref{eqn:NN_cap_def} are constraints on the parameters of the neural network.

The restriction in \eqref{eqn:NN_cap_def} on $|D^\eta \calU|$ implies that we are interested in Lipschitz neural networks. Recent machine learning research emphasizes that Lipschitz neural networks ensure robustness of a solution to new or unforeseen data \cite{fazlyab2019efficient, MR4207504, pauli2021training, MR1618535}.

Using the version provided by \cite[Theorem~3.1]{Hornick90}, the Universal Approximation Theorem asserts that $\calN\calN$ is dense in $\calC^{1,1}([0,T] \times \calP([d]))$ with the norm $\|\cdot\|_{\calC^{1,1}}$ defined in \eqref{eqn:main_norm}. We also define:
\begin{align}\label{eqn:dbme_dense_space}
    \calC^{1,1} ([d] \times \calP([d]) ; C_0, C_1) := \big\{ \calV \in \calC^{1,1} ([d] \times \calP([d])) : |\calV| \leq C_0, |D^\eta \calV|\leq C_1\big\}.
\end{align} Since we are only concerned with the master equation's solution, which is a Lipschitz function with Lipschitz derivatives (see Proposition~\ref{prop:me_results}), it will be appropriate at times to deal with this restricted set.

\section{The DBME Algorithm and Convergence Results}\label{sec:dbme}

Having given the necessary background, we now present a deep-learning-based backward algorithm for learning the master equation termed Deep Backward Master Equation (DBME).

\subsection{The DBME Algorithm \ref{alg:DBME}}\label{sec:dbme_details}
First, we introduce some notation relevant to the algorithm. Then, we justify some assumptions made on the neural networks under consideration, and introduce an auxiliary Kolmogorov forward ODE, before introducing the DBME algorithm itself. 

The time interval $[0,T]$ will be partitioned into $\pi := \{t_0=0, t_1, \dots, t_{N-1}, t_N =T\}$, with increments denoted $\Delta t_i := t_{i+1} - t_i$. Define \[
|\pi|:= \max_{i\in [N]} \Delta t_i.
\] Throughout we assume that there exist $\check c, \hat c>0$ such that: 
\begin{align}\label{eqn:pi_proportional}
\frac{\check c}{N} \leq |\pi| \leq \frac{\hat c}{N};
\end{align} that is, $|\pi|^{-1}$ is roughly proportional to $N$. For functions and random variables at time $t_i \in \pi$, we will often use only a subscript $i$ as opposed to $t_i$ for brevity.

For the time $t_i\in \pi$, we approximate the master equation's solution  $U(t_i, x, \eta)$ by $\calU_i \in \calN\calN$ with input dimension $\delta_1=d+2$. Its parameters are denoted by $\theta^i$ and we use the notation $\calU_i(x,\eta; \theta^i)$. We will sometimes use the shorthand notation $\hat\calU_i(x,\eta) := \calU_i(x,\eta; \hat\theta^i)$, where the vector of parameters $\hat\theta^i$ optimizes a loss function defined in the sequel. We will not use a neural network at the last time in the partition and simply set $\hat \calU_N(x,\eta) := g(x,\eta)$.

Let $L=2$ for instance and $\bar\delta = (\delta_2 (d+3))$. For the vector of parameters $\theta \in \R^{\bar\delta}$, denote the parameters from the scaling inside the activation function by $\check\theta$. We can write a neural network $\calU_i(x,\eta ; \theta ) \in \calN\calN$ as:
\begin{align}\label{eqn:nn_structure}
    \calU_i (x,\eta ; \theta) := \sum_{k \in [\delta_2]} b_k \phi(a^{0,k} x + a^k \cdot \eta + c^k),
\end{align} 
where $a^k := (a^{y,k})_{y \in [d]} \in \R^d$, $a^{0,k} \in \R$, $b\in\R^{\delta_2}$, $c^k \in \R$ for all $k=1,\dots, \delta_2$, and these parameters make up the entries in $\theta \in \R^{\bar\delta}$ and $\check \theta := (a^k)_{k=1}^{\delta_2}$. In the case $L=2$, we note that $\check \theta$ is a $(\delta_2 \times 2)$-matrix with real entries and denote $\|\cdot\|_2$ as the matrix operator $2$-norm. Per \cite[Section~4]{MR4207504}, we know that when $\|\check \theta\|_2 < K$, some positive real number, the Lipschitz constant of the network with parameters $\theta\in\R^{\bar\delta}$ is no more than $K$ (possibly after modifying $K$ to account for the Lipschitz constant of the activation function $\phi$). So, for the rest of this section, which focuses on the DBME algorithm, we define the following set of parameters:
\begin{equation}
    \label{eq:def-Theta-NN}
    \Theta := \{\theta \in \R^{\bar\delta} \mid \|\check \theta\|_2 \leq 2C_{L,U}\},
\end{equation}
where recall that $C_{L,U}$ is the Lipschitz constant of $U$ depending only on the problem parameters that can be computed from \cite[Section~3]{cec-pel2019}. Thus any network $\calU(\cdot, \cdot;\theta)$ where $\theta\in\Theta$ is Lipschitz with constant no more than $2C_{L,U}$.

\begin{remark}\label{rmk:nn_bound}
    Previously in \eqref{eqn:NN_cap_def}, we introduced the class of neural networks $\calN\calN(C_0,C_1)$ as appropriate for approximating $U$. In this remark, we elaborate further. By \eqref{eqn:U_bound}, the master equation solution is uniformly bounded by the constant $\tilde T$, given in \eqref{eqn:U_bound}. So, we are justified in truncating the neural network $\calU_i$ by replacing $\calU_i$ with $\calU_i \wedge \tilde T$. This ensures that $\calU_i$ will always be uniformly bounded. Using a result of \cite[Lemma~3.1]{CZ2022} and the boundedness of $U$ implies: 
    \[
    |\Delta_x U(t,x,\eta)| \leq \sqrt{2d} \max_{y\in [d]}|U(t,y,\eta)| \leq \sqrt{2d}\tilde T.
    \] 
    Because of the restriction of $\calU_i$ to $\calU_i\wedge \tilde T$, the same kind of computation holds for $\calU_i$ as in the previous display. Consequently, \eqref{eqn:interval_restriction} holds and the local regularity assumptions on the Hamiltonian $H$ given in \eqref{eqn:H_regularity} are in force with the neural network as an argument; namely the derivatives $D_pH$ and $D_{pp}^2H$ are Lipschitz and there exists a positive constant $C>0$ such that:
    \[
    D^2_{pp} H(x,\Delta_x \calU_i(\cdot, \eta)) \leq -C,
    \] for all $(x,\eta) \in [d]\times \calP([d])$ and all $i\in [N]$. The local regularity assumptions will be useful in the proof of algorithmic convergence. The upshot of this remark is that for any $\theta^i \in \Theta$,
    \[
    \calU_i(\cdot,\cdot;\theta^i)\wedge \tilde T \in \calN\calN(\tilde T, 2C_{L,U}).
    \] 
\end{remark}

One final element to introduce before presenting the algorithm is the solution of a forward Kolmogorov's equation. This equation mirrors the flow of measures $\mu$ in the MFG equilibrium. However, instead of utilizing the master equation's solution $U$ in the dynamics, we employ a neural network. This solution is a key element in the algorithm's loss function. 
In details: 
for any $\kappa \in \calP([d])$, any $i \in \{0,\dots,N-1\}$, and any $\theta\in\Theta$, define $M^\theta_{i}(\kappa) := \rho_i^{\theta,\kappa}(t_{i+1})$, where $\rho_i^\theta(t) :=(\rho_i^\theta(t,x))_{x\in[d]}$ uniquely solves the (forward) ODE:
\begin{align}
\begin{split}\label{eqn:M_theta_i_def}
    &\frac{d}{dt} \rho_i^{\theta,\kappa}(t,x)=\sum_{y\in[d]}\rho_i^{\theta,\kappa}(t,y)\gamma^*_x(y,\Delta_y  \calU_{i}(\cdot,\rho^{\theta,\kappa}_i(t);\theta)), \qquad (t,x)\in[t_i,t_{i+1}]\times[d],\\
    &\rho_i^{\theta,\kappa}(t_i) = \kappa.
\end{split}
\end{align} The previous display is Kolmogorov's forward equation over the time interval $[t_i,t_{i+1}]$ with initial distribution $\kappa \in \calP([d])$, and which corresponds to a non-linear Markov chain with rate matrix $\gamma^*_x(y,\Delta_y \calU_i(\cdot, \rho^{\theta, \kappa}_i(t) ; \theta^i))$. Note that the control's input is the difference $\Delta_y \calU_{i}$ of the neural network $\calU_i$, rather than that of $U$; in light of \eqref{eqn:consistency_relation}, this is similar to the MFG system \eqref{eqn:mfg}. Moreover, $\rho^{\theta, \kappa}_i$ appears within the rate matrix, which makes the Markov chain non-linear. The equation \eqref{eqn:M_theta_i_def} is well-posed because the neural network and $\gamma^*$ are Lipschitz. The notation $M^\theta_i (\kappa)$ can be thought of as the terminal distribution of this non-linear Markov chain at time $t_{i+1}$; that is, the effect of propagating the initial distribution $\kappa$ by a control using neural network values, parameterized by $\theta \in \Theta$.

\begin{algorithm}
\caption{DBME}\label{alg:DBME}
\begin{algorithmic}[1]
\State {\bf Input:} 
A vector of initial parameters $\boldsymbol{\theta} := (\theta^i)_{i=0}^{N-1}$. 
\State {\bf Output:} A grid of neural networks $(\calU_i)_{i=0,\dots,N}$ approximating the solution to \eqref{master_equation:finite_horizon} on $\pi$. 

\State $\hat\calU_N \gets g$
\State $i \gets N - 1$
\State $\calU_i$ initialized via $\theta^i$.
\For{$i$ from $N-1$ to $0$}
    \State Recalling \eqref{eqn:M_theta_i_def}, compute:
    $\displaystyle
    \hat \theta^i \in \argmin\limits_{\theta^i \in \Theta} L_i(\theta^i)
    $
    \State where
\begin{align}\label{eqn:loss_def}
    L_i(\theta^i) := \max\limits_{(x,\kappa)\in [d]\times \calP([d])} \Big| \hat\calU_{i+1} (x, M^{\theta^i}_i(\kappa)) - \calU_i(x, \kappa; \theta^i)+ (\Delta t_i) \bar H(x, \kappa, \Delta_{x} \calU_i(\cdot, \kappa;\theta^i))\Big|
\end{align} 
\State $\hat \calU_i (\cdot, \cdot) \gets \calU_i(\cdot, \cdot; \hat\theta^i) \wedge \tilde T$ 
\EndFor
\end{algorithmic}
\end{algorithm}

The next remark motivates the algorithm's structure and the form of the loss function \eqref{eqn:loss_def} in the algorithm.
\begin{remark}
From \eqref{eqn:mfg} and \eqref{eqn:U_def}--\eqref{eqn:consistency_relation} it follows that:
\begin{align*}
    U(t_{i+1},x,\mu^{t_i,\kappa}(t_{i+1}))-U(t_{i},x,\kappa)+(t_{i+1}-t_i)\bar H(x,\kappa, \Delta_x U(t_i,\cdot,\kappa))\approx 0
\end{align*} when $\Delta t_i = t_{i+1}-t_i$ is small. Recall that the dynamics of $(\mu^{t_i,\kappa}(t))_{t\in[t_i,t_{i+1}]}$, the flow of measures under equilibrium from \eqref{eqn:mfg}, depends on the master equation solution, making it unsuitable for the loss function. To address this, in the loss function, we replace $\mu^{t_i,\kappa}(t_{i+1})$ by $M^{\theta^i}_i(\kappa)$, defined through equation~\eqref{eqn:M_theta_i_def}, which is similar to the equation for $\mu^{t_i,\kappa}(t)$ in~\eqref{eqn:mfg} but with the neural network appearing in its dynamics instead of the function $U$. 
Notice that, while the loss uses the forward ODE~\eqref{eqn:M_theta_i_def}, it does not explicitly make use of the backward ODE. This is because the master equation captures the value function, see~\eqref{eqn:consistency_relation}.

Next, note that the algorithm goes backward in time. The value at terminal time $T$ is known. Then at time $t_{i}$, assuming that we have an estimation for $U(t_{i+1},x,\eta
)$ for any $(x,\eta)\in[d]\times\calP([d])$, we go backward in time. In the analysis to follow,
we use the structure of the loss function and \eqref{eqn:M_theta_i_def} to define a forward-backward system that resembles the MFG system \eqref{eqn:mfg} up to a small error that emerges from the loss function. It is the comparison between these two systems that yields the convergence result of Theorem \ref{thm:DBME}.
\end{remark}

\subsection{Remarks for the Practitioner}

We now make a few remarks on the algorithm that are relevant for its implementation.

\begin{remark}\label{rem52}
    The reader may wonder how to ensure that the vector of parameters $\theta^i$ obtained at each step in Algorithm \ref{alg:DBME} lies in the set $\Theta$ defined in~\eqref{eq:def-Theta-NN}. One solution proposed by \cite[Algorithm~2]{MR4207504} is to project the weights at each step of stochastic gradient descent. Projection guarantees $\theta^i \in \Theta$ and that the network obtained is Lipschitz. However, given a large enough upper bound $kC_{L,U}$, where previously in \eqref{eqn:nn_structure} we chose $k=2$, the projection will not occur in practice and the algorithm can essentially run normal SGD. The result from \cite{MR4207504} is preferred here because of its theoretical guarantees; however many other methods may be used to improve the regularity of a neural network, for instance dropout layers, learning rate schedulers, random initialization, and so on. 
\end{remark}

\begin{remark}\label{rmk:appx_max}
    In practice, we cannot compute a maximum over infinitely many elements $\kappa \in \calP([d])$ as we write in Algorithm \ref{alg:DBME}. Therefore in the implementation, we uniformly generate samples $\kappa$ in $\calP([d])$ and compute the propagation $M^\theta_i(\kappa)$ for each before minimizing the loss by SGD. In what follows we remark on the order of this error due to the approximation. Denote the set of all samples by $\calK$. Define:
    \begin{align*}
        \tilde L_i(\kappa;\theta^i) := \max_{x\in [d]} | \hat\calU_{i+1} (x, M^{\theta}_i(\kappa)) - \calU_i(x, \kappa; \theta^i)+ (\Delta t_i) \bar H(x, \kappa, \Delta_{x} \calU_i(\cdot, \kappa;\theta^i))|,
    \end{align*} so that $\tilde L_i$ is related to the algorithm's loss $L_i$ as:
    \begin{align*}
        L_i(\theta^i) = \max_{\kappa \in \calP([d])}\tilde L_i(\kappa ; \theta^i).
    \end{align*} 
For any $i=0,\ldots,N$, and $\theta^i$, choose the following:
\begin{align*}
\hat\kappa \in \argmax_{\kappa\in\calK}\tilde L_i(\kappa;\theta^i).
\end{align*}

\begin{lemma}[Approximation of Sampled Loss to Theoretical]\label{lemma:sample_conv}
    With $\hat\kappa$, the sampled maximum, as defined above, denote the cardinality of $\calK$ by $|\calK|$. Then, there exists a constant $C_d>0$ depending on the parameters of the problem and the dimension $d$ such that:
    \begin{align*}
        \EE |L_i(\theta^i) - \tilde L_i(\hat\kappa; \theta^i)| \leq C_d |\calK|^{-1/(d-1)}. 
    \end{align*}
\end{lemma} The proof, together with an explicit expression for $C_d$, can be found in the appendix.

\end{remark}

\begin{remark}
    Note that the maximum is in general not smooth and yet we aim to eventually take the gradient of the loss function which requires the loss to be differentiable. In practice, the maximum is computed as a smooth maximum; that is, a smooth function that approximates a maximum. For prior work applying stochastic gradient descent to a maximized loss function, and comments on its robustness, see \cite{ShwartzW16}.
\end{remark}

\begin{remark} 
It is important to note that computing $\bar H$ explicitly is only possible when the Hamiltonian $H$ is computable or at least approximable. This is true for instance in both numerical examples that we consider later in Section \ref{sec:numerics}. In greater generality, when the running cost function is convex in its control argument, this amounts to solving a convex optimization problem at each time step.
\end{remark}

\subsection{DBME Main Results}

Recalling the definition of $\hat\theta^i$ from \eqref{eqn:loss_def}, define:
\begin{align}\label{eqn:epsilons}
\begin{split}
    \eps_i :=  L_i(\hat \theta^i)\qquad\text{and}\qquad\eps^{\calU} :=  \max_{i\in [N]} \eps_i.
\end{split}
\end{align} 
Note that $\eps_N = 0$ because in the algorithm, we set $\calU_N := g$.  Nonetheless we include $\eps_N$ in \eqref{eqn:epsilons} in order to use the notation $[N]$ instead of $[N-1]$. 



We can now state our two main results with respect to the DBME. Each result is, loosely speaking, converse to the other. The first theorem says that there exist neural networks close to the true solution in our desired class, and when a neural network is close to the true function, the loss from the algorithm is small. 

We recall that $\tilde{T}$ is defined in~\eqref{eqn:U_bound} 
    and that $C_g, C_f$ and $C_F$ are bounds on the maximum norms of $g, f$ and $F$ respectively, and $C_{L,U}$ is a bound on the Lipschitz constant of $U$.

\begin{theorem}[DBME Approximation]\label{thm:dbme_existence}
    Set: 
    \begin{align}\label{eqn:c0c1}
        C_0 := 2 \tilde{T} \qquad\text{ and }\qquad C_1 := 2 C_{L,U}.
    \end{align}  
    For every $\eps >0$, there exist neural networks $(\calU_i^\eps)_{i\in [N]} \subset \calN\calN(C_0,C_1)$  with parameters $(\theta^{i,\eps})_{i\in [N]}$ such that for all $(i,x,\kappa) \in [N]\times [d] \times \calP([d])$:
    \begin{equation}\label{eqn:prop_e}
        |U(t_i,x,\kappa) - \calU_i^\eps(x,\kappa; \theta^{i,\eps})| < \eps.
    \end{equation} 
    Moreover, for any neural network satisfying \eqref{eqn:prop_e}, the DBME algorithm's loss is bounded as:
    \begin{equation}\label{eqn:loss_bound}
        \max_{i\in [N]} L_i(\theta^i) < 2\eps + C|\pi|,
    \end{equation}
    where we recall that $C$ is a constant depending only on the model parameters.
\end{theorem}

The second theorem provides an upper bound on the difference between the neural networks and the true master equation solution in terms of the algorithm's error. Intuitively, it says that any network trained by the algorithm that minimizes the loss is in fact close to the true master equation solution. 

\begin{theorem}[DBME Convergence]\label{thm:DBME} 
There exist $N_0>0$ and $C>0$ depending only on the problem's data with the following property. For every $N>N_0$ and for every $i=0,\dots,N-1$, letting $\hat\calU_i = \calU_i(\cdot,\cdot; \hat\theta^i)$ denote a neural network with parameters $\hat\theta^i$ minimizing the loss  in the DBME Algorithm \ref{alg:DBME}, we have:
\begin{align*}
    \max_{(i, x, \eta) \in[N]\times [d] \times \calP([d])} |U(t_i,x,\eta) - \hat\calU_i(x, \eta)| \leq C \left(\tfrac{1}{N} + N^{} \eps^{\calU}\right).
\end{align*}
\end{theorem}


\begin{remark}
    We note that the order of this error---that is, $\calO(N^{-1})$ in one term and $\calO(N \eps^{\calU})$ in the other---is like the error bound from \cite{MR4081911}, but with a smaller order of $N$ in the second term. The larger order of $N$ in \cite{MR4081911} likely comes from their use of random propagation in the algorithm and corresponding fixed point system. In our case, the propagation is deterministic and the error is defined differently. 
\end{remark}

\begin{remark}
    To minimize the right hand side of the theorem, one first selects the number $N$ of time steps large enough to minimize the left error term. Then one selects the number $\bar\delta$ of parameters large enough so that $\eps^{\calU}$ is small enough. Notice that the definition of $\eps^{\calU}$ takes into account the approximation error but not the training error (which stems from the optimization method used, such as SGD or Adam). However, we can expect that with a large enough number of training steps, the realized error can also be made small enough.
\end{remark}

\subsection{Auxiliary Lemma and Proof}

In this section we prove a measure bound lemma that comes into play in the proofs of Theorem~\ref{thm:dbme_existence} and Theorem \ref{thm:DBME}. For every $i\in\{0,\ldots,N-1\}$, we denote $(u^\kappa_i, \mu^{\kappa}_i) = (u^{t_i,\kappa}, \mu^{t_i,\kappa})$ as the solution to \eqref{eqn:mfg} on $[t_i,T]$ with $\mu^{\kappa}_i(t_i) = \kappa \in \calP([d])$. In the following sections, we write $\|\cdot\|_{i,\iy}$ to mean the sup-max norm taken over all $t\in [t_i, t_{i+1}]$ and all $x\in [d]$, that is: $\|\varphi\|_{i,\iy} = \sup_{t\in [t_i, t_{i+1}]} \max_{x \in [d]}|\varphi(t,x)|$, where $\varphi: [0,T] \times [d] \to \mathbb{R}$.

\begin{lemma}[Measure Bound]\label{lemma:duality_bound} There exists $C>0$ depending only on the problem data such that for all $|\pi|$ small enough and every $\theta\in\Theta$ and $(i,\kappa) \in [N] \times \calP([d])$, 
\begin{align}\label{eqn:sus_rho}
    \|\rho_i^{\theta, \kappa} - \mu^{\kappa}_i \|_{i,\iy} \leq \frac{C|\pi|}{1- C|\pi|} \| \calU_i(\cdot,\rho_i^{\theta, \kappa}(\cdot) ; \theta)  - U(\cdot, \cdot, \mu^\kappa_i(\cdot))\|_{i,\iy}.
\end{align} 
\end{lemma}

\begin{proof}
Recall the definition of the loss function in \eqref{eqn:loss_def}. 
For every $\theta \in \Theta$, define the mapping $e_i(\cdot,\cdot;\theta):[d]\times\calP([d])\to\R$ by: 
\begin{align}\label{eq:e_i}
    e_i(x,\kappa ; \theta) := -\hat\calU_{i+1} (x,M^{\theta}_i(\kappa)) + \calU_i(x, \kappa ; \theta) - (\Delta t_i) \bar H(x,\kappa, \Delta_{x} \calU_i(\cdot,\kappa ; \theta)).
\end{align} 
Recalling \eqref{eqn:M_theta_i_def}, $M^{\theta}_i(\kappa) = \rho_i^{\theta, \kappa}(t_{i+1})$, \eqref{eq:e_i} rewrites: 
\begin{align}
\label{eq:U_i}
        &\hat\calU_{i+1}(x,\rho_i^{\theta,\kappa}(t_{i+1})) - \calU_i(x,\kappa ; \theta) + (\Delta t_i)\bar H (x,\kappa,\Delta_x \calU_i(\cdot,\kappa ; \theta))+e_i(x,\kappa ; \theta)=0.
\end{align} 
Note that we initialized the DBME algorithm with $\hat \calU_N(x,\eta)=g(x,\eta)$; also, going backwards with $i=N-1,\ldots,0$, on each time interval $[t_i,t_{i+1}]$, we may think about \eqref{eq:U_i} as a  backward finite-difference equation, with $\hat\calU_{i+1}(x,\rho_i^{\theta,\kappa}(t_{i+1}))$ as the terminal condition.
We refer to  \eqref{eq:U_i} and \eqref{eqn:M_theta_i_def} together as a forward-backward system.


The difference $\rho_i - \mu^\kappa$ solves an ODE that after integrating yields: for every $s \in [t_i, t_{i+1}]$, 
\begin{align*}
    \rho_i^{\theta, \kappa}(s,x) - \mu^\kappa_i(s,x) &= \int_{t_i}^{s} \Big[\sum_{y\in [d]} (\rho_i^{\theta, \kappa}(t,y) - \mu^\kappa_i(t,y)) \gamma^*_x(y,\Delta_y \calU_i(\cdot,\rho_i^{\theta,\kappa}(t) ; \theta))  \\
    &\qquad\qquad\quad + \sum_{y\in [d]} \mu^\kappa_i(t,y) \big[ \gamma^*_x(y,\Delta_y \calU_i(\cdot,\rho_i^{\theta, \kappa}(t) ; \theta)) - \gamma^*_x(y,\Delta_y u^\kappa_i(t,\cdot))\big] \Big]dt .
\end{align*} 
Since $\gamma^*$ is bounded and locally Lipschitz (see Remark~\ref{rem:gammastar-Lip}), for every $s \in [t_i, t_{i+1}]$ 
\begin{align*}
    |\rho_i^{\theta, \kappa}(s,x) - \mu^\kappa_i(s,x)| &\leq \int_{t_i}^s \Big[ dC_{\gamma^*} \max_{y\in [d]}  |\rho_i^{\theta, \kappa}(t,y) -\mu^\kappa_i(t,y)| \\
    &\qquad\qquad + dC_{L,\gamma^*} \max_{y\in [d]} |\Delta_y (\calU_i(\cdot,\rho_i^{\theta, \kappa}(t) ; \theta) -  u^\kappa_i(t,\cdot))| \Big] dt, 
\end{align*} 
where we recall that $C_{\gamma^*}, C_{L,\gamma^*}$ depend only on the model parameteres and not on the neural network or the time discretization. Taking the supremum over all $(t,x) \in [t_{i},t_{i+1}] \times [d]$ for the last term in the above display and a maximum over the state in the left hand side yields:
\begin{align*}
    \max_{y\in [d]}|\rho_i^{\theta, \kappa}(s,y) - \mu^\kappa_i(s,y)| &\leq C \int_{t_i}^s \max_{y\in [d]}  |\rho_i^{\theta, \kappa}(t,y) -\mu^\kappa_i(t,y)| dt \\
    &\quad + C |\pi| \|\calU_i(\cdot,\rho_i^{\theta, \kappa}(\cdot) ; \theta) -  u^\kappa_i(\cdot,\cdot))\|_{i,\iy} . 
\end{align*} 
Again taking another supremum bound with respect to $s$ and combining like terms, and then using the consistency property \eqref{eqn:consistency_relation}, we obtain \eqref{eqn:sus_rho}.

\end{proof}

\subsection{Proof of Theorem~\ref{thm:dbme_existence}} 
We recall that $C$ is a generic, positive constant that depends only on the problem's parameters and in particular is independent of $\eps$ and $|\pi|$.  We allow $C$ to change from one line to the next as needed.

Recall the parameters $C_0$ and $C_1$ from \eqref{eqn:c0c1}. We notice that $U \in \calC^{1,1} ([d] \times \calP([d]) ; C_0/2, C_1/2)$ (see \eqref{eqn:dbme_dense_space}). Indeed, by \eqref{eqn:U_bound}, $\|U\|_\infty \leq \tilde T = C_0/2.$
Due to Proposition \ref{prop:me_results}, we know that $U$ is Lipschitz with Lipschitz constant $C_{L,U}$ (in fact, $D^\eta U$ is Lipschitz). So, $\|D^\eta U\|_\infty \le C_1/2.$
Now, by \cite[Theorem~3.1]{Hornick90}, the Universal Approximation Theorem asserts that $\calN\calN$ is dense in $\calC^{1,1}([0,T] \times \calP([d]))$: for any $\eps>0$, there exists $\calU \in \calN\calN$ such that $\|U - \calU\|_{\calC^{1,1}} < \epsilon$. This neural network satisfies: $\|\calU\|_\infty \le C_0/2 + \eps$ and $\|D^\eta \calU\|_\infty \le C_1/2 + \eps$. For $\eps< \min\{C_0/2, C_1/2\}$, we deduce that $\|\calU\|_\infty \le C_0$ and $\|D^\eta \calU\|_\infty \le C_1$, so it belongs to the class $\calN\calN(C_0,C_1)$. 

Next, we show \eqref{eqn:loss_bound}. Since Lemma \ref{lemma:duality_bound} is proved in greater generality, we will use it here for $\theta = \theta^{i,\eps}$.  We will denote \begin{align}\label{eqn:e_hat}
    e_i^\eps (x, \kappa) := e_i(x,\kappa;\theta^{i,\eps}).
\end{align}

Let $(x,\kappa) \in [d] \times \calP([d])$. Let $s \in [t_i, t_{t+1}]$. By integrating \eqref{eqn:mfg} and by~\eqref{eq:U_i}, along with \eqref{eqn:consistency_relation},
\begin{equation}
\label{eq:U-hatU-expression}
\begin{split}
    U(s,x,\mu^\kappa_i(s)) - \calU^\eps_i(x,\kappa) &= U(t_{i+1},x,\mu^\kappa_i(t_{i+1})) - \calU^\eps_{i+1}(x,\rho_i^{\theta^{i,\eps}, \kappa}(t_{i+1})) \\
        &\quad + \int_s^{t_{i+1}}  \bar H (x,\mu^\kappa_i(t), \Delta_x U(t,\cdot,\mu^\kappa_i(t))) dt \\
        &\quad- \Delta t_i \bar H(x,\kappa, \Delta_x \calU^\eps_i(\cdot,\kappa))  \\
        &\quad + e_i^\eps(x,\kappa).
\end{split}
\end{equation} 
Solving for $e_i^\eps$ above, and using triangle inequality, yields:
\begin{align*}
    |e_i^\eps(x,\kappa)| &\leq |U(s,x,\mu^\kappa_i(s)) - \calU^\eps_i(x,\kappa)| \\
        &\quad + |\calU^\eps_{i+1}(x,\rho_i^{\theta^{i,\eps}, \kappa}(t_{i+1})) - U(t_{i+1},x,\mu^\kappa_i(t_{i+1}))|\\
        &\quad + \Big|\int_s^{t_{i+1}} \big[ \bar H (x,\mu^\kappa_i(t), \Delta_x U(t,\cdot,\mu^\kappa_i(t)))\big]dt - \Delta t_i \bar H(x,\kappa, \Delta_x \calU^\eps_i(\cdot,\kappa))\Big|.
\end{align*}
We are going to bound each term of the right hand side, using the Lipschitz continuity of $\calU^\eps_{i+1},$ the triangle inequality again, Lemma~\ref{lemma:duality_bound} and the Lipschitz continuity of $\calU^\eps_i$, the fact (by integrating in time the forward equation and using the fact that $U$ is bounded) that $|\mu^{\kappa}_i (s) - \kappa| \leq C|\pi|$, the boundedness of $\bar H$, and \eqref{eqn:prop_e}.

For the first term,
\begin{align*}
    |U(s,x,\mu^\kappa_i(s)) - \calU^\eps_i(x,\kappa)|
    &\le |U(s,x,\mu^\kappa_i(s)) - \calU^\eps_i(x,\mu^\kappa_i(s))|
        + |\calU^\eps_i(x,\mu^\kappa_i(s)) - \calU^\eps_i(x,\kappa)|
    \\
    &\le \eps + C C_{L,\calU^\eps_i}|\pi| .
\end{align*}
For the second term,
\begin{align*}
    &|\calU^\eps_{i+1}(x,\rho_i^{\theta^{i,\eps}, \kappa}(t_{i+1})) - U(t_{i+1},x,\mu^\kappa_i(t_{i+1}))|
    \\
    &\le |\calU^\eps_{i+1}(x,\rho_i^{\theta^{i,\eps}, \kappa}(t_{i+1})) - \calU^\eps_{i+1}(x,\mu^\kappa_i(t_{i+1}))|
        + |\calU^\eps_{i+1}(x,\mu^\kappa_i (t_{i+1})) - U(t_{i+1},x,\mu^\kappa_i(t_{i+1}))|
    \\
    &\le C_{L,\calU^\eps_{i+1}}|\rho_i^{\theta^{i,\eps}, \kappa}(t_{i+1}) - \mu^\kappa_i(t_{i+1})| 
        + |\calU^\eps_{i+1}(x,\mu^\kappa_i (t_{i+1})) - U(t_{i+1},x,\mu^\kappa_i(t_{i+1}))|
    \\
    &\le C_{L,\calU^\eps_{i+1}}|\rho_i^{\theta^{i,\eps}, \kappa}(t_{i+1}) - \mu^\kappa_i(t_{i+1})|
        + |\calU^\eps_{i+1}(x,\mu^\kappa_i (t_{i+1})) - U(t_{i+1},x,\mu^\kappa_i(t_{i+1}))|
    \\
    &\le C C_{L,\calU^\eps_{i+1}} |\pi| + \eps.
\end{align*}
For the third term, 
\begin{align*}
    &\Big|\int_s^{t_{i+1}} \big[ \bar H (x,\mu^\kappa_i(t), \Delta_x U(t,\cdot,\mu^\kappa_i(t)))\big]dt - \Delta t_i \bar H(x,\kappa, \Delta_x \calU^\eps_i(\cdot,\kappa))\Big|
    \le 2C_{\bar H}|\pi|.
\end{align*}
Combining these bounds, we obtain:
\[
    |e_i^\eps(x,\kappa)| \le C(\eps + |\pi|).
\]

Now we can take a maximum over $(x,\kappa) \in [d]\times \calP([d])$, and with \eqref{eqn:loss_def}, this yields \eqref{eqn:loss_bound}.

\qed


\subsection{Proof of Theorem \ref{thm:DBME}}

Fix $\eta \in \calP([d])$. Recall \eqref{eqn:M_theta_i_def}. For any $i\in[N]$ and any $\hat\theta^i$ minimizing $L_i$ in~\eqref{eqn:loss_def}, set
\[
    \rho_i:=\rho^{\hat\theta^i,\mu^\eta(t_i)}_i.
\]
Note that for any $\theta^i$, 
\begin{align*}
    L_i(\theta^i)
    &=\max_{x,\kappa} e_i(x,\kappa; \theta^i) 
    \\
    &\ge \max_{x\in [d]} e_i(x, \mu^n(t_i); \theta^i)
    \\
    &= \max_{x\in [d]} \big|\hat\calU_{i+1} (x,M^{\theta^i}_i(\mu^\eta (t_i))) - \calU_i(x,\mu^\eta (t_i); \theta^i) + (\Delta t_i) \bar H(x,\mu^\eta (t_i), \Delta_{x} \calU_i(\cdot,\mu^\eta (t_i);\theta^i))\big|.
\end{align*} 
So when $\theta^i = \hat \theta^i$, 
\begin{align}\label{eqn:pre_inequality_eps}
    \eps_i \geq \max_{x\in [d]} \big|\hat\calU_{i+1} (x,M^{\hat\theta^i}_i(\mu^\eta (t_i)) ) - \hat \calU_i(x,\mu^\eta (t_i)) + (\Delta t_i) \bar H(x,\mu^\eta (t_i), \Delta_{x} \hat \calU_i(\cdot,\mu^\eta (t_i)))\big|,
\end{align} 
where we recall $\eps_i$ is defined in \eqref{eqn:epsilons}. Like in \eqref{eqn:pre_inequality_eps}, by the definition of the loss \eqref{eq:e_i} and \eqref{eqn:e_hat}, $\|e_i^\eps\|_\iy = \eps_i$.

Once more integrating \eqref{eqn:mfg},  by~\eqref{eq:U_i}, and using \eqref{eqn:consistency_relation},
\begin{equation}
\label{eq:U-hatU-expression-eta}
\begin{split}
    U(s,x,\mu^\eta(s)) - \hat\calU_i(x,\mu^\eta(t_i)) &= U(t_{i+1},x,\mu^\eta(t_{i+1})) - \hat\calU_{i+1}(x,\rho_i(t_{i+1})) \\
        &\quad + \int_s^{t_{i+1}} \big[ \bar H (x,\mu^\eta(t), \Delta_x U(t,\cdot,\mu^\eta(t)))\big]dt \\
        &\quad- \Delta t_i \bar H(x,\mu^\eta(t_i), \Delta_x \hat\calU_i(\cdot,\mu^\eta(t_i)))  \\
        &\quad + \hat e_i(x,\mu^\eta(t_i)).
\end{split}
\end{equation} 
Furthremore, by the Lipschitz continuity of $H$,
\begin{align*}
        &\Big| \int_{s}^{t_{i+1}} \big[ \bar H (x,\mu^\eta(t), \Delta_x U(t,\cdot,\mu^\eta(t)))\big] dt - \Delta t_i \bar H(x,\mu^\eta(t_i), \Delta_x \hat\calU_i(\cdot,\mu^\eta(t_i))\Big|  \\
        &\leq C_{L,\bar H}\Delta t_i \|\mu^\eta - \mu^\eta(t_i)\|_{i,\iy}
        + C_{L,\bar H} C\Delta t_i\|U(\cdot,\cdot, \mu^\eta(\cdot)) - \hat\calU_i(\cdot,\mu^\eta(t_i))\|_{i,\iy}.
\end{align*} 
Going back to~\eqref{eq:U-hatU-expression-eta},
 taking the supremum over $(s,x) \in [t_i,t_{i+1}] \times [d]$, using the fact that $\|\mu^\eta(\cdot) - \mu^\eta(t_i)\|_{i,\iy} \leq C \Delta t_i$, and combining like terms,
\begin{align*}
    &\|U(\cdot,\cdot, \mu^\eta(\cdot)) - \hat\calU_i(\cdot,\mu^\eta(t_i))\|_{i,\iy} \\
    &\qquad\leq \tfrac{1}{1-C|\pi|} \Big[ \max_{y\in [d]} |U(t_{i+1},y,\mu^\eta(t_{i+1})) - \hat\calU_{i+1}(y,\rho_i(t_{i+1}))|  + C|\pi|^2 + \max_{y\in [d]}|\hat{e}_i(y,\mu^\eta(t_i))|\Big].
\end{align*} Note that the above inequality requires $C|\pi|<1$, which is true for $N$ large enough, as we assume in the statement. 
Using the triangle inequality,
\begin{equation}\label{eq:U-hatUi-triangle}
\begin{split}
&\|U(\cdot,\cdot,\mu^\eta(\cdot)) - \hat\calU_i(\cdot,\mu^\eta(t_i))\|_{i,\iy}\\ 
    &\quad\leq \tfrac{1}{1- C|\pi|} \Big[C|\pi|^2 + \max_{y\in [d]}|U(t_{i+1},y,\mu^\eta(t_{i+1})) - \hat\calU_{i+1} (y,\mu^\eta(t_{i+1}))|  \\
        &\qquad\qquad\qquad + \max_{y\in [d]}|\hat\calU_{i+1} (y,\mu^\eta(t_{i+1})) - \hat\calU_{i+1} (y,\rho_i(t_{i+1}))|  + \max_{y\in [d]}|\hat{e}_i(y,\mu^\eta(t_i))|\Big].
\end{split}
\end{equation}

Next, we will use the fact that $\hat\calU_{i+1}$ is Lipschitz, \eqref{eqn:sus_rho}, triangle inequality, the fact that $\hat\calU_i$ is Lipschitz, that $\|\rho_i(\cdot) - \mu^\eta(\cdot)\|_{i,\iy}\leq C|\pi|$, and the fact that 
\[
    \text{the map }\qquad
    (0,|\pi|^{-1})\ni C\mapsto\frac{C}{1-C|\pi|} \qquad\text{is nondecreasing},
\] 
in order to update the generic constant $C>0$.

We note that:
\begin{align*}
    &\max_{y\in [d]}|\hat\calU_{i+1} (y,\mu^\eta(t_{i+1})) - \hat\calU_{i+1} (y,\rho_i(t_{i+1}))|
    \\
    &\le C_{L, \hat \calU_{i+1}} |\mu^\eta(t_{i+1}) - \rho_i(t_{i+1} )| 
    \\
    &\le \tfrac{C|\pi|}{1-C|\pi|} \|\hat\calU_i(\cdot,\rho_i(\cdot)) - U(\cdot,\cdot, \mu^\eta(\cdot))\|_{i,\iy} 
    \\
    &\le \tfrac{C|\pi|}{1-C|\pi|}\|\hat\calU_i(\cdot,\rho_i(\cdot)) - \hat\calU_i(\cdot,\mu^\eta(t_i))\|_{i,\iy} 
      + \tfrac{C|\pi|}{1-C|\pi|} \|U(\cdot,\cdot,\mu^\eta(\cdot)) - \hat\calU_i(\cdot,\mu^\eta(t_i))\|_{i,\iy} 
    \\
    &\le \tfrac{C|\pi|^2}{1-C|\pi|}C_{L,\hat\calU_i}
     + \tfrac{C|\pi|}{1-C|\pi|} \|U(\cdot,\cdot,\mu^\eta(\cdot)) - \hat\calU_i(\cdot,\mu^\eta(t_i))\|_{i,\iy},
\end{align*}
where we used the fact that $C_{L,\hat\calU_i}$ is bounded by a constant depending only on the model's parameters.

Going back to~\eqref{eq:U-hatUi-triangle} and using the above bound, we obtain:
\begin{align*}    
    &\|U(\cdot,\cdot,\mu^\eta(\cdot)) - \hat\calU_i(\cdot,\mu^\eta(t_i))\|_{i,\iy}
    \\ 
    &\quad\leq \tfrac{1}{1- C|\pi|} \Big[C|\pi|^2 + \max_{y\in [d]}|U(t_{i+1},y,\mu^\eta(t_{i+1})) - \hat\calU_{i+1} (y,\mu^\eta(t_{i+1}))|  
    \\
    &\qquad\qquad\qquad + \max_{y\in [d]}|\hat\calU_{i+1} (y,\mu^\eta(t_{i+1})) - \hat\calU_{i+1} (y,\rho_i(t_{i+1}))|  + \max_{y\in [d]}|\hat{e}_i(y,\mu^\eta(t_i))|\Big]
    \\
    &\leq \tfrac{C}{(1-C|\pi|)^2} \Big[|\pi|^2 +  |\pi| \|U(\cdot,\cdot,\mu^\eta(\cdot)) - \hat\calU_i(\cdot,\mu^\eta(t_i))\|_{i,\iy}\Big] \\
        &\quad + \tfrac{1}{1-C|\pi|}\Big[ \max_{y\in [d]}|\hat e_i(y,\mu^\eta(t_i))| + \max_{y\in [d]}|U(t_{i+1},y,\mu^\eta(t_{i+1})) - \hat\calU_{i+1} (y,\mu^\eta(t_{i+1}))| \Big].
\end{align*}

For the rest of the proof, we denote by $\Gamma_0$ the constant $C$ appearing in the above right-hand side. That is, we are fixing $\Gamma_0$ to be $C$ at this particular point in the proof, but $C$ later on will be generic as usual. 
Rearranging terms, there exists $\Gamma_1>0$ depending only on $\Gamma_0$ and $\pi$ such that: 
\begin{align}
\begin{split}\label{eqn:step_1_sus}
    &\|U(\cdot,\cdot,\mu^\eta(\cdot)) - \hat\calU_i(\cdot,\mu^\eta(t_i))\|_{i,\iy} \\&\quad\leq \big(\tfrac{1}{1-\Gamma_0|\pi|}\big)\big(\tfrac{(1-\Gamma_1|\pi|)^2}{1-3\Gamma_1|\pi|}\big) \Big[ \Gamma_0|\pi|^2 +  \max_{y\in [d]}|\hat e_i(y,\mu^\eta(t_i))| \\
    &\hspace{4.5cm}+ \max_{y\in [d]} |U(t_{i+1},y,\mu^\eta(t_{i+1})) - \hat\calU_{i+1} (y,\mu^\eta(t_{i+1}))|\Big].
\end{split}
\end{align} 
Taking $C_2= 3\Gamma_1$, we have for all $|\pi|$ small enough that:
\[
    \big(\tfrac{1}{1-\Gamma_0|\pi|}\big)\big(\tfrac{(1-\Gamma_1|\pi|)^2}{1-3\Gamma_1|\pi|}\big) \leq \big(\tfrac{1-\Gamma_1 |\pi|}{1-C_2|\pi|}\big)^2=:Q.
\] 
Note that: 
\begin{align*}
    \sum_{j=1}^{N-1}Q^j &= \tfrac{Q^N - Q}{Q-1} \\
    &= \big[\big(\tfrac{1-\Gamma_1|\pi|}{1-C_2|\pi|}\big)^{2N} - Q\big] \tfrac{(1-C_2|\pi|)^2}{2|\pi|(C_2 - \Gamma_1) + |\pi|^2(\Gamma_1^2-C_2^2)}\\
    &\leq \big(\tfrac{1-\Gamma_1|\pi|}{1-C_2|\pi|}\big)^{2N}\cdot   \tfrac{(1-C_2|\pi|)^2}{2|\pi|(C_2 - \Gamma_1) + |\pi|^2(\Gamma_1^2-C_2^2)}.
\end{align*} Recalling \eqref{eqn:pi_proportional}, and noting that:
\[
    \big(\tfrac{1-\Gamma_1|\pi|}{1-C_2|\pi|}\big)^{2N} \to e^{2C_2 - 2\Gamma_1} = e^{4\Gamma_1},
\] 
monotonically as $N\to \iy$, we have for all $|\pi|>0$ small enough that:
\begin{align}\label{eqn:iterative_sum}
    \sum_{j=1}^{N-1}Q^j \leq CN e^{4\Gamma_1}. 
\end{align} Returning to \eqref{eqn:step_1_sus}, we know that for $i = N-1$, the terminal conditions agree and therefore,
\begin{align*}
    \|U(\cdot,\cdot,\mu^\eta(\cdot)) - \hat\calU_{N-1}(\cdot,\mu^\eta(t_{N-1})))\|_{N-1,\iy} &\leq Q\big[C|\pi|^2 +  \max_{y\in [d]}|\hat e_{N-1}(y,\mu^\eta(t_i))|\big].
\end{align*} 
By backward induction with \eqref{eqn:step_1_sus} and then using \eqref{eqn:iterative_sum}, 
\begin{align*}
    \max_{i\in [N]} \|U(\cdot,\cdot,\mu^\eta(\cdot)) - \hat\calU_{i}(\cdot,\mu^\eta(t_i)))\|_{i,\iy} &\leq \Big(\sum_{j=1}^{N-1}Q^j\Big) \big[C|\pi|^2 +  \max_{i\in[N]} \|\hat e_i\|_\iy \big] \\
    &\leq C|\pi| + CN \eps^{\calU}. 
\end{align*} This inequality holds for all $\eta \in \calP([d])$ and therefore we can take a maximum over all $\eta \in \calP([d])$ to conclude the proof.

\qed

\section{The DGME Algorithm and Convergence Results}\label{sec:dgm}

In what follows, we describe the DGME adaptation of the DGM to solve the master equation, and we then provide two novel theoretical results. Recall that $U$ denotes the unique classical solution to~ \eqref{master_equation:finite_horizon} and that $U(\cdot, x, \cdot)\in\calC^{1,1}_{}([0,T]\times\calP([d]))$ for all $x\in [d]$. By some previous remarks and from \cite[Theorem~3.1]{Hornick90}, it is known that $\calN\calN$ is dense in $(\calC^{1,1}([0,T] \times \calP([d])), \|\cdot\|_{\calC^{1,1}})$. 
 
 In Algorithm \ref{alg:DGM_ME}, we introduce the DGME, modified with the loss to a maximum from the expectation of a squared quantity. We also note that unlike the DBME, in which we restricted our attention to $\theta\in\Theta$, in the DGME we can take $\theta\in\R^{\bar\delta}$ without restriction.

\begin{algorithm}
\caption{DGME}\label{alg:DGM_ME}
\begin{algorithmic}[1]
\State {\bf Input:} An initial vector $\theta$. 
\State {\bf Output:} A trained vector $\hat \theta$ such that $\calU(\cdot, \cdot, \cdot ; \hat\theta)$ approximately solves \eqref{master_equation:finite_horizon}.
    \State Compute:
    $\displaystyle 
    \hat \theta \in \argmin\limits_{\theta \in \R^{\bar\delta}} L(\theta)
    $
    \State where 
    \begin{align}\label{eqn:dgm_max_loss}
        \begin{split}
            &L(\theta) := \max\limits_{(t,x,\eta)\in [0,T)\times [d]\times \calP([d])}\Big\{ \big| \pl_t \calU (t,x,\eta ; \theta) + \bar H(x,\eta,\Delta_x \calU(t,\cdot,\eta ; \theta)) \\
        &\qquad\qquad\qquad\qquad\qquad\qquad\qquad + \sum_{y\in [d]} \eta_y D^\eta_y \calU (t,x,\eta ; \theta)\cdot\gamma^*(y,\Delta_y \calU (t,\cdot,\eta ; \theta)) \big|\\
        &\qquad\qquad\qquad\qquad\qquad\qquad\qquad\qquad + |\calU (T,x,\eta ; \theta) - g(x,\eta) |\Big\}
        \end{split}
    \end{align} 
\end{algorithmic}
\end{algorithm}
As we noted in Remark \ref{rmk:appx_max}, we cannot, in practice, compute the maximum over the infinitely many elements in $[0,T) \times [d] \times \calP([d])$. So, like for the DBME and the original DGM formulation, we randomly sample the space instead. As in Remark \ref{rmk:appx_max}, we note that uniformly sampling the space results is an estimator for the maximum where for $\calK$ samples, we get $\calO(|\calK|^{-1/(d-1)})$ error for the estimator. 
The master equation we study in this paper only involves first-order derivatives with respect to the inputs $t$ and $\eta$. Hence computing the gradient with respect to the neural network parameters of the loss function~\eqref{eqn:dgm_max_loss} only requires computing second-order derivatives, which can be done by automatic differentiation. 
\footnote{If a second-order term was involved in the PDE, then the computation of the gradient of the loss would require third-order derivatives. Using automatic differentiation would probably be prohibitive in high dimension. In such cases, our method could perhaps be modified, using approximate second-order derivative computations, as proposed in~\cite[Section 3]{MR3874585}.} 

The following two theorems establish convergence of the DGME and are analogous to results from \cite[Theorem~7.1, Theorem~7.3]{MR3874585}, but these results do not apply to the master equation due to its distinct form. Instead, we use the structure of the MFG system itself in order to obtain useful bounds. Yet another difference is that our results are given in the supremum norm while those of \cite{MR3874585} are in $L^2$ norm.

In the theorem below, we note that the density of the neural networks implies there always exists a neural network uniformly close to the master equation solution $U$; then, we write that any close network satisfies the master equation \eqref{master_equation:finite_horizon} in an approximate sense. As we mentioned in the introduction, the two main results of each section are, loosely speaking, converse to one another. The first result deals with existence of networks close to the true solution and shows that when a network is close, the corresponding algorithm's loss is small.

\begin{theorem}[DGME Approximation]\label{prop:existence}
For every $\eps>0$, there exists a neural network $\hat\calU \in \calN\calN$ such that:
\begin{equation}\label{eqn:dgm_density}
    \|\hat\calU - U\|_{\calC^{1,1}} < \eps.
\end{equation} 
Moreover, let $e^{\hat\calU}:[0,T] \times [d] \times \calP([d]) \to \R$ be defined by: 
\begin{align}\label{psi:approximate_me_sol}
\begin{split}
     e^{\hat\calU}(t,x,\eta) &:= \sum_{y,z\in [d]} \eta_y D^\eta_{yz} \hat\calU(t,x,\eta) \gamma_z^*(y,\eta,\Delta_y \hat\calU(t,\cdot,\eta)) + H(x,\eta,\Delta_x \hat\calU(t,\cdot,\eta)) + \pl_t \hat\calU(t,x,\eta), \\
    e^{\hat\calU}(T,x,\eta) &:= \hat\calU(T,x,\eta)-g(x,\eta).
\end{split}
\end{align}
Then $e^{\hat\calU}$ is measurable and there exists a positive constant $\CDGME$, depending only on the problem data and independent of $\eps$, such that:
\begin{align}\label{eqn:error_conds}
    \begin{split}
    \|e^{\hat\calU}\|_\iy &< \CDGME\eps.
    \end{split}
\end{align} 

\end{theorem}


We stress that Theorem~\ref{prop:existence} has two parts. The first part, namely~\eqref{eqn:dgm_density}, holds by universal approximation theorem and is straightforward given the (known) regularity of the master equation solution. It is not directly related to the algorithm. The second part, namely~\eqref{eqn:error_conds}, provides a bound on the residual of the PDE (including the terminal condition) when using a neural network which gives~\eqref{eqn:dgm_density}. This provides an upper bound on the loss function of Algorithm~\ref{alg:DGM_ME}, which is simply the PDE residual.

We now present the main DGME result concerning convergence. This theorem asserts that if the DGME finds a network with small error in its loss function, then that network must be uniformly close to the master equation's solution $U$.

\begin{theorem}[DGME Convergence]\label{thm:dgm_convergence}
Let $0<\eps < 1$. Let $\hat\calU : [0,T]\times [d] \times \calP([d])\to\R$ be a neural network obtained by the DGME such that \eqref{psi:approximate_me_sol}--\eqref{eqn:error_conds} hold. Then, there exists $C>0$ depending only on the problem data and independent of $\eps$ such that: 
\begin{equation}\label{eqn:thm:dgm_bound}
\|\hat\calU - U\|_{\iy} < C\eps.
\end{equation} Consequently, any sequence of neural networks $(\hat\calU_n)_{n\in\N}$, with corresponding errors $(\eps_n)_{n\in\N}$ such that $\eps_n \to 0$, converges to $U$ uniformly.
\end{theorem}
The proof is based on a careful analysis of the master equation, and its connection with forward-backward ODE systems characterizing the MFG equilibrium for each initial condition. To the best of our knowledge, the result cannot be deduced from the literature. In particular, the analysis of the DGM in~\cite{MR3874585} is only for a specific class of (quasilinear, parabolic) PDEs, which does not cover ours. Furthermore, the master equation, despite some similarities, is not an HJB equation and does not enjoy a comparison principle. 

\begin{remark}
    The attentive reader would observe that in the proof of Theorems \ref{prop:existence} and \ref{thm:dgm_convergence}, we effectively establish the existence of two constants, $\underline{C},\bar{C}>0$, such that for any $\hat\calU\in\calN\calN$, one has
    $\underline{C} \|e^{\hat\calU}\|_\infty \le \|\hat\calU - U\|_{\iy} \le \bar{C} \|e^{\hat\calU}\|_\infty$.
\end{remark}

\subsection{Proof of Theorem \ref{prop:existence}}

We note that \eqref{eqn:dgm_density} is immediate from the universal approximation result of \cite[Theorem~3.1]{Hornick90}. It suffices to prove \eqref{psi:approximate_me_sol}--\eqref{eqn:error_conds}. Recall that $C>0$ is a generic constant independent of $t$ and $\eta$, and may change from one line to the next. 


From \eqref{eqn:dgm_density}, for any $\eps>0$ and any $x\in [d]$, there exists a neural network $\hat\calU(\cdot,x,\cdot)\in\calN\calN$ such that:
\[
\|U(\cdot,x,\cdot) - \hat\calU(\cdot,x,\cdot)\|_{\calC^{1,1}} < \eps,\qquad\text{for all $x\in [d]$.}
\] Since \eqref{eqn:dgm_density} holds for all $x\in[d]$, we may write
$
    \|\hat\calU - U\|_{\calC^{1,1}}$ to mean $ \max_{x\in [d]}\|\hat\calU(\cdot,x,\cdot) - U(\cdot,x,\cdot)\|_{\calC^{1,1}}$ without ambiguity. 
Note that:
\[
\|\hat\calU\|_{\iy} \leq \|U\|_\iy + \eps \le \tilde{T} + \eps,
\] 
so $\hat\calU$ is bounded above by the right hand side of \eqref{eqn:interval_restriction}. When $\eps$ is small enough, $H$ is restricted to the interval $[-W,W]$ and note that $H$ is locally Lipschitz on $[-W,W]$ by assumption. We may therefore treat $H$ as Lipschitz. 

Recall the definition of $e^{\hat\calU}(t,x,\eta)$ from \eqref{psi:approximate_me_sol}. For brevity, we write $e^{\hat\calU}$ as $e$ during the proof. Since $U$ solves~\eqref{master_equation:finite_horizon}, we find that:
\begin{equation}
\label{eq:DGM-e-hatpsi-equation}
\begin{split}
    e(t,x,\eta) &= \pl_t \hat\calU(t,x,\eta) - \pl_t U(t,x,\eta) + \bar H(x,\eta,\Delta_x \hat\calU(t,\cdot,\eta))  - \bar H(x,\eta, \Delta_x U(t,\cdot,\eta)) \\
    &\quad + \sum_{y,z\in [d]} \left[\eta_y D^\eta_{yz} \hat\calU (t,x,\eta) \gamma_z^*(y,\Delta_y \hat\calU (t,\cdot,\eta)) - \eta_y D^\eta_{yz} U (t,x,\eta) \gamma_z^*(y,\Delta_y U (t,\cdot,\eta))\right]. 
\end{split}
\end{equation} 
We can check that, for the last line of \eqref{eq:DGM-e-hatpsi-equation}: 
\begin{align*}
    &\Big|\sum_{y,z\in [d]} \left[\eta_y D^\eta_{yz} \hat\calU (t,x,\eta) \gamma^*_z(y,\Delta_y \hat\calU(t,\cdot,\eta)) - \eta_y D^\eta_{yz} U (t,x,\eta) \gamma^*_z(y,\Delta_y U(t,\cdot,\eta))\right] \Big| \\
    &\quad\leq d^2 \max_{y,z\in [d]} |D^\eta_{yz} \hat\calU (t,x,\eta) \gamma^*_z(y,\Delta_y \hat\calU(t,\cdot,\eta)) -  D^\eta_{yz} U (t,x,\eta) \gamma^*_z(y,\Delta_y U(t,\cdot,\eta))| \\
    &\quad \leq d^2 \max_{y,z\in [d]} |D^\eta_{yz} \hat\calU (t,x,\eta) \gamma^*_z(y,\Delta_y \hat\calU (t,\cdot,\eta)) -  D^\eta_{yz} U (t,x,\eta) \gamma^*_z(y,\Delta_y \hat\calU(t,\cdot,\eta))| \\
    &\quad\quad + d^2 \max_{y,z\in [d]} | D^\eta_{yz} U (t,x,\eta)  \gamma^*_z(y,\Delta_y \hat\calU(t,\cdot,\eta)) - D^\eta_{yz} U(t,x,\eta) \gamma^*_z(y,\Delta_y U(t,\cdot,\eta))|.
\end{align*} 
Using the facts that $\gamma^*$ is bounded and Lipschitz, that the classical master equation solution $U$ has $D^\eta U$ bounded, and \eqref{eqn:dgm_density}, we can bound the above quantity by:
\[
    d^2 C_{\gamma} \max_{y,z\in[d]}|D^\eta_{yz} (\hat\calU(t,x,\eta) - U(t,x,\eta))| + d^2 C_{D^\eta U} C_{L,\gamma} \max_{y\in [d]} |\Delta_y (\hat\calU (t,\cdot, \eta) - U(t,\cdot,\eta))| < C\eps,
\] 
where $C$ is independent of $t,\eta$. Going back to~\eqref{eq:DGM-e-hatpsi-equation}, using the Lipschitz continuity of $H$ and of $F$, and again \eqref{eqn:dgm_density}, we obtain: 
    $|e(t,x,\eta)| < C\eps.$
In other words, $\hat\calU$ satisfies \eqref{psi:approximate_me_sol}--\eqref{eqn:error_conds}. 
\qed

In the following section, we prove a lemma that is at the heart of Theorem \ref{thm:dgm_convergence}. The lemma allows us to circumvent the unfortunate fact that the master equation has no comparison principle; thus we introduce an approximate MFG system and use duality.

\subsection{Approximate MFG System}

In the following lemma, we define the function $v$ through the network $\hat\calU$ and we use the fact that it satisfies an approximate Hamilton--Jacobi--Bellman equation. Its evolution is coupled with a function $\lambda$, which solves a forward Kolmogorov equation of a non-linear Markov chain. That is, define $\lambda$ as the solution to:
\[
\frac{d}{dt} \lambda_x(t) = \sum_{y\in [d]} \lambda_y(t) \gamma_x^*(y,\Delta_y \hat\calU(t,\cdot,\lambda(t))), \quad \lambda(t_0) = \eta,
\] on $t\in [t_0,T]$. Note that $\lambda$ is well-defined by the Picard--Lindel\"of Theorem.

\begin{lemma}[Approximate MFG System and Network Bound]\label{lemma:network_bound}
    Let $0<\eps < 1$. Let $\hat\calU : [0,T]\times [d] \times \calP([d])\to\R$ be a neural network obtained by the DGME. Let the error function $e$ be defined as in~\eqref{psi:approximate_me_sol}. Assume \eqref{eqn:error_conds} holds. Then, by setting $v_x(t):=\hat\calU(t,x,\lambda(t))$, we get that $(v, \lambda)$ solves the approximate MFG system on $(t,x) \in [t_0,T) \times [d]$: 
\begin{align}
\begin{split}\label{eqn:mfg_appx}
    &\frac{d}{dt} v_x (t) + \bar H(x, \lambda(t), \Delta_x v(t)) = e(t,x,\lambda(t)), \\ 
    &\frac{d}{dt} \lambda_x(t) = \sum_{y\in [d]} \lambda_y(t) \gamma^*_x(y,\Delta_y v (t)), \\ 
    &\lambda(t_0) = \eta, \\ 
    &v_x(T) = g(x,\lambda(T)) + e(T,x,\lambda(T)).
\end{split}
\end{align} 
Moreover, for all $(t,x,\eta) \in [t_0,T]\times [d] \times \calP([d])$ and when $\eps$ is small enough:
    \[
    |\hat\calU (t,x,\eta)| \leq T[C_f + C_F + \CDGME\eps] + C_g + \CDGME\eps \leq W,
    \] 
    where $\CDGME$ is defined in~\eqref{eqn:error_conds} and $W$ is defined in \eqref{eqn:interval_restriction}. Hence, the regularity assumptions on $H$ given in \eqref{eqn:H_regularity} apply. 
\end{lemma}

\begin{remark}
    The careful reader may wonder why the bound on $|\hat\calU|$ above does not immediately follow from Theorem \ref{prop:existence}. This is because Lemma \ref{lemma:network_bound} and Theorem \ref{thm:dgm_convergence} deal with a network obtained from the DGME, while Theorem \ref{prop:existence} does not. Theorem \ref{prop:existence} shows that running the DGME is not hopeless since there exist neural networks with the properties we are seeking. Offhand, we only know that the network considered in Lemma \ref{lemma:network_bound} and Theorem \ref{thm:dgm_convergence} satisfies the master equation in some approximate way; therefore, we use only this fact to prove $|\hat\calU|$ is bounded.  
\end{remark}

\subsection{Proof of Lemma \ref{lemma:network_bound}} We use the structure of the MFG system to obtain important estimates. That is, we fix $t_0\in [0,T)$ and $\eta \in\calP([d])$ and set $\lambda$ to be the unique solution to the Kolmogorov equation in \eqref{eqn:mfg_appx} on $[t_0, T]$.

For $v_x(t) := \hat\calU(t,x,\lambda(t))$ defined on $[t_0,T]$, we use the fact that $\hat\calU$ satisfies \eqref{psi:approximate_me_sol} to get:
\begin{align*}
    \frac{d}{dt} v_x(t) &= \frac{d}{dt} \hat\calU (t,x,\lambda(t)) \\
    &=\pl_t \hat\calU (t,x,\lambda(t)) + \sum_{y\in [d]} \lambda_y(t) D^\eta_y \hat\calU (t,x,\lambda(t)) \cdot \gamma^*(y,\Delta_y \hat\calU(t,\cdot, \lambda(t))) \\
    &= e(t,x,\lambda(t)) - \bar H(x,\lambda(t), \Delta_x v(t)),
\end{align*}
where the error function $e$ is defined in~\eqref{psi:approximate_me_sol}. 
So together, $(v, \lambda)$ solves \eqref{eqn:mfg_appx}. It remains to show the bound on $\hat\calU$.

Let $\tilde\calX$ be a jump process on $[d]$ with rate matrix given by $\gamma_x^* (y,\Delta_y v (t))$; namely, its infinitesimal generator is given by $[\gamma_x^* (y,\Delta_y v (t))]_{x,y \in [d]}$. Set $\tilde\calX_{t_0} = z\in [d]$.  By an application of It\^o's lemma, and using the boundary conditions with \eqref{psi:approximate_me_sol}, we have  for any $t \in [t_0,T]$, 
\begin{align*}
    &\EE \big[ g(\tilde \calX_T, \lambda(T)) + e(T,\tilde \calX_T, \lambda(T)) - v_z (t_0)\big] 
    \\
    &\quad= \EE \int_{t_0}^T \big[\pl_t v_{\tilde \calX_s}(s) + \gamma^*(\tilde \calX_s, \Delta_{\tilde \calX_s} v(s)) \cdot \Delta_{\tilde \calX_s} v (s) \big] ds \\ 
    &\quad= \EE \int_{t_0}^T \big[ -f(\tilde \calX_s, \gamma^*(\tilde \calX_s, \Delta_{\tilde \calX_s} v(s))) - F(\tilde \calX_s, \lambda(s))  + e(s,\tilde \calX_s,\lambda(s)) \big] ds.
\end{align*} 
Then,
\[
    v_z(t_0) = J(t_0, x, \lambda, \gamma^*) + \EE \Big[\int_{t_0}^T e(t,\tilde \calX_t, \lambda(t)) dt + e(T,\tilde \calX_T, \lambda(T))\Big].
\] 
So by definition of $v$, by~\eqref{eqn:error_conds}, since $f$ and $F$ are bounded, and since $t_0$ was arbitrary,
\begin{align}\label{eqn:U_unif_bound}
    |\hat\calU(t,x,\eta)| \leq T[C_f + C_F +  \CDGME\eps] + C_g + \CDGME\eps.
\end{align}

By \eqref{eqn:U_unif_bound}, $\hat\calU$ is uniformly bounded. Moreover, since $\eps < 1$, $\hat\calU(t,x,\eta) \in [-W,W]$ for all $(t,x,\eta)$ ($W$  is defined in~\eqref{eqn:interval_restriction}). On the input $\hat\calU(t,x,\eta)$ then, \eqref{eqn:H_regularity} holds.
\qed

In particular, from Lemma \ref{lemma:network_bound}, we have that for all $(t,x,\eta) \in [0,T]\times [d]\times \calP([d])$: \[
D^2_{pp} H(x,\Delta_x \hat\calU(t,\cdot,\eta)) \leq -C_{2,H}.
\] This observation will be useful in the following proof.

\subsection{Proof of Theorem \ref{thm:dgm_convergence}}

The proof relies on an approximate MFG system on the time interval $[t_0,T]$ with initial distribution $\eta\in\calP([d])$. We compare its solution to the true solution using MFG duality, that is, integrating the backward equation's solution $u$ as a test function in the forward equation to obtain estimates on the solution. The MFG solution is related to the master equation as discussed in Section \ref{sec:me_recall}, and since $t_0$ and $\eta$ are arbitrary, we obtain convergence to the master equation solution. 

The proof proceeds in three steps. Already we formulated the approximate MFG system~\eqref{eqn:mfg_appx} defined through a neural network and meant to mimic the MFG. We then proved a bound on the solution to the approximate system. 
In the subsequent steps, we will take advantage of the fact that the approximate MFG system is structured like the MFG system. In Step 1, we use the duality of the MFG and approximate MFG systems to obtain several estimates. That is, we treat the approximate MFG value $v$ as a test function to integrate
against, in order to derive our estimates. In Steps 2 and 3, we integrate the differences of solution and combine these estimates with duality to finish the proof. 



\subsubsection*{Step 1: Duality}
Fix $\eta \in \calP([d])$ and let $(u^{t_0,\eta} , \mu^{t_0,\eta})$ solve \eqref{eqn:mfg} on $[t_0, T]$ with $\mu^{t_0,\eta}(t_0) = \eta$. For ease of notation, we suppress the superscript and simply write $(u,\mu)$ instead. Define $\scrQ := v - u$ and $\scrM := \lambda- \mu$.  Using \eqref{eqn:mfg} and \eqref{eqn:mfg_appx}, $(\scrQ, \scrM)$ solves: 
\begin{align}
    \begin{split}\label{eqn:qm}
        -&\frac{d\scrQ_x}{dt}(t) = \bar H(x,\lambda(t), \Delta_x v(t)) - \bar H(x,\mu (t), \Delta_x u(t)) - e(t,x, \lambda(t)), \\
        &\frac{d\scrM_x}{dt}(t) = \sum_{y\in [d]} \big[\lambda_y(t) \gamma^*_x(y,\Delta_y v(t)) - \mu_y(t) \gamma^*_x(y,\Delta_y u(t))\big], \\
        &\scrQ_x(T) = g(x,\lambda(T)) - g(x,\mu(T)) + e(T,x, \lambda(T)), \\ 
        &\scrM (t_0) = 0.
    \end{split} 
\end{align} 
 Using the product rule and \eqref{eqn:qm}:
\begin{align*}
    \sum_{x\in [d]}\scrQ_x(T)\scrM_x(T) 
    &= \sum_{x\in [d]} \int_{t_0}^T \Big\{\scrM_x(t)[-\bar H(x,\lambda(t), \Delta_x v(t)) + \bar H(x,\mu(t), \Delta_x u(t))] \\
    &\qquad\qquad\qquad + \scrQ_x(t)\sum_{y\in [d]}\left[ \lambda_y(t) \gamma^*_x(y,\Delta_y v(t)) -\mu_y(t) \gamma^*_x(y,\Delta_y u(t))\right]\\
    &\qquad\qquad\qquad - \scrM_x(t)e(t,x, \lambda(t))\Big\} dt.  
\end{align*} 
We use the definition of $\bar H$, the terminal condition on $\scrQ$, and the fact that the rates sum to zero (namely, for any $x\in[d]$ and $p\in\R^d$, $\sum_y \gamma^*_y(x,\Delta_x p)=0$) to get: 
\begin{align*}
    \sum_{x\in [d]} (g(x,\lambda(T)) &- g(x,\mu(T))) \scrM_x(T) + \sum_{x\in [d]} \int_{t_0}^T  [F(x, \lambda(t)) - F(x, \mu(t))]\scrM_x(t) dt \\
    &=  \int_{t_0}^T \Big[\sum_{x\in [d]} [H(x,\Delta_x u(t)) - H(x,\Delta_x v(t))] \scrM_x(t) \\
    &\quad + \Delta_x \scrQ(t) \cdot [\lambda_x(t) \gamma^*_y(x,\Delta_x v(t)) - \mu_x(t) \gamma^*_y(x,\Delta_x u(t))]_{y\in [d]}\Big] dt \\
    &\quad - \int_{t_0}^T e(t,\cdot,\lambda(t)) \cdot \scrM(t) dt - e(T,\cdot,\lambda(t)) \cdot \scrM(T).
\end{align*} 
By assumption, $g$ and $F$ are Lasry--Lions monotone and hence the left-hand side of the above display is non-negative. Hence,
\begin{align}\label{eqn:pre_second_order}
    0 &\leq  \int_{t_0}^T \Big[\sum_{x\in [d]} [H(x,\Delta_x u(t)) - H(x,\Delta_x v(t))] \scrM_x(t) \\ \notag
    &\qquad\qquad + \Delta_x \scrQ(t) \cdot [\lambda_x(t) \gamma^*_y(x,\Delta_x v(t)) - \mu_x(t) \gamma^*_y(x,\Delta_x u(t))]_{y\in [d]}\Big] dt \\ \notag
    &\quad - \int_{t_0}^T e(t,\cdot,\lambda(t)) \cdot \scrM(t) dt - e(T,\cdot,\lambda(t)) \cdot \scrM(T).
\end{align}  By Lemma \ref{lemma:network_bound}, we may apply \eqref{eqn:DH} and \eqref{eqn:H_regularity}. That is, there exists $C_{2,H}>0$ such that for all $t\in [t_0,T]$ and all $x\in [d]$: 
\begin{align*}
    &H(x,\Delta_x v(t)) - H(x,\Delta_x u(t)) - \Delta_x \scrQ(t) \cdot \gamma^*(x,\Delta_x v(t)) \leq -C_{2,H} |\Delta_x \scrQ(t)|^2, \\
    &H(x,\Delta_x u(t)) - H(x,\Delta_x v(t)) + \Delta_x \scrQ(t) \cdot \gamma^*(x,\Delta_x u(t)) \leq -C_{2,H} |\Delta_x \scrQ(t)|^2.
\end{align*} Using the above display in \eqref{eqn:pre_second_order} and rearranging terms,
\begin{align}
\begin{split}\label{eqn:appx_dual}
    &C\int_{t_0}^T \Big[\sum_{x\in [d]} |\Delta_x \scrQ(t)|^2 (\lambda_x(t) + \mu_x(t))\Big] dt \\
    &\qquad\leq - \int_{t_0}^T e(t,\cdot,\lambda(t)) \cdot \scrM(t) dt - e(T,\cdot,\lambda(T)) \cdot \scrM(T),
\end{split}
\end{align} 
where recall that $C$ denotes a generic positive constant and its value may change from one line to the next. 

\subsubsection*{Step 2: Integrating the backward equation from \eqref{eqn:qm}} Integrating the backward equation from \eqref{eqn:qm}, using the Lipschitz continuity of $F$ and $H$ (the latter uses Lemma \ref{lemma:network_bound}, again), using the terminal condition for $\scrQ$, using the Lipschitz continuity of $g$, and taking the supremum norm:
\begin{align*}
        |\scrQ_x(t_0)| &\leq |\scrQ_x(T)| + \int_{t_0}^T \big[|F(x,\lambda(t)) - F(x,\mu(t))| + |H(x,\Delta_x v(t)) - H(x,\Delta_x u(t))| \\
        &\qquad\qquad\qquad\quad\qquad + |e(t,x,\lambda(t))|\big] dt \\
        &\leq |\scrQ_x(T)| + C \int_{t_0}^T \big[|\lambda(t) - \mu(t)| + |\Delta_x (v(t) - u(t))| + |e(t,x,\lambda(t))|\big] dt \\ 
        &= |\scrQ_x(T)| + C \int_{t_0}^T \big[|\scrM(t)| + |\Delta_x \scrQ(t))| + |e(t,x,\lambda(t))|\big] dt \\
        &\leq |g(x,\lambda(T)) - g(x,\mu(T))| + |e(T,x,\lambda(T))| \\
            &\qquad + C \int_{t_0}^T \big[|\scrM(t)| + |\Delta_x \scrQ(t)| + |e(t,x,\lambda(t))|\big] dt \\ 
        &\leq C\Big(\|\scrM\|_{\iy} + \|e\|_{\iy} + \int_{t_0}^T \max_{z\in [d]}|\scrQ_z(t)| dt\Big).
\end{align*} Then, by Gronwall's inequality and taking the supremum over all $t_0 \leq T$: 
\begin{align}
    \begin{split}\label{eqn:q_bound}
        \|\scrQ\|_\iy \leq C(\|\scrM\|_{\iy} + \|e\|_{\iy}).
    \end{split}
\end{align} 

\subsubsection*{Step 3: Integrating the forward equation from \eqref{eqn:qm}} Integrating the measure equation from \eqref{eqn:qm} on $[t_0,t_1]$, adding and subtracting a term, and using the boundedness of $\gamma^*$, and the Lischitz continuity of $\gamma^*$ (over the range to which the inputs belongs):
\begin{align*}
    |\scrM_x(t_1)| &\leq \sum_{y\in [d]} \int_{t_0}^{t_1} \big[|\scrM_y(t) \gamma_x^*(y,\Delta_y v(t))| + |\mu_y(t) (\gamma_x^*(y,\Delta_y v(t)) - \gamma_x^*(y,\Delta_y u(t))) |\big]dt \\
    &\leq C \sum_{y\in [d]} \int_{t_0}^{t_1} \big[ |\scrM_y(t)| + \mu_y(t)|\Delta_y \scrQ(t)| \big]dt \\
    &\leq C \int_{t_0}^{t_1} \max_{y\in [d]} |\scrM_y(t)|dt + C \sum_{y\in [d]}\int_{t_0}^{t_1} \mu_y(t)|\Delta_y \scrQ(t)| dt.
\end{align*} 
Using Jensen's inequality and Gronwall's inequality: 
\begin{align}
    \|\scrM\|_\iy \leq C \int_{t_0}^{t_1} \sqrt{\sum_{y\in [d]} \mu(t,y) |\Delta_y \scrQ(t,\cdot)|^2} dt.
\end{align}  In the above display, using Jensen's inequality, \eqref{eqn:appx_dual}, and a supremum bound imply, 
\begin{align*}
    \|\scrM\|_{\iy} &\leq  C \int_{t_0}^{t_1} \sqrt{\sum_{y\in [d]} \mu_y(t) |\Delta_y \scrQ(t)|^2} dt\\
    &\leq C \Big(\int_{t_0}^{t_1} \sum_{y\in [d]} \mu_y(t) |\Delta_y \scrQ(t)|^2 dt\Big)^{1/2} \\
    &\leq C\Big(\int_{t_0}^{T} |e(t,\cdot,v(t)) \cdot \scrM(t)| dt + |e(T,\cdot,v(T)) \cdot \scrM(T)|\Big)^{1/2}\\
    &\leq C\big(\|e\|_{\iy}\|\scrM\|_{\iy} + \|e\|_{\iy}\|\scrM\|_{\iy}\big)^{1/2}\\
    &\leq C  \|e\|_{\iy}^{1/2}\|\scrM\|_{\iy}^{1/2}.
\end{align*} So by Young's inequality,
\begin{align*}
    \|\scrM\|_\iy \leq C\|e\|_{\iy}.
\end{align*} Combining this estimate with \eqref{eqn:q_bound}, and using \eqref{eqn:error_conds},
we obtain:
\begin{align*}
    \|\scrQ\|_\iy + \|\scrM\|_\iy \leq C\|e\|_\iy.
\end{align*} 
The above reasoning holds for every $(t_0,\eta)$. Recalling the definitions of $\scrQ$, $v$, and \eqref{eqn:U_def}, we obtain \eqref{eqn:thm:dgm_bound}.
\qed

\section{Numerical experiments}\label{sec:numerics}

In this section, both the DBME and DGME algorithms were implemented in Python using TensorFlow 2. Both programs were run on the Great Lakes computing cluster, a high-performance computing cluster available for University of Michigan research. All algorithms were run on the cluster's standard nodes, each of which consists of thirty-six cores per node. We expect further improvement in run-time given a practitioner switches to a specialized, GPU-enabled node. We also note that, in the original formulation of the DGM, \cite{MR3874585} use an LSTM-like architecture while we are using a vanilla, fully-connected, feed-forward architecture. For all networks featured, we used four layers of sixty nodes each, with sigmoid activation function, excluding the input and output layers. The output layers used ELU. Code for the DGME and DBME algorithms, data for models used, and code used to create the visualizations in this section can be found on GitHub at \url{https://github.com/ethanzell/DGME-and-DBME-Algorithms}. 

\subsection{Example of Quadratic Cost} A classical example in MFG literature is that of quadratic costs. In this section, we adapt the examples of \cite[Example~1]{cec-pel2019} and \cite[Example~3.1]{bay-coh2019} to solve the corresponding master equation using the DGME and DBME algorithms. 

Namely, we set the running costs and terminal cost, respectively: 
\begin{align*}
    &f(x,a) := b\sum_{y\neq x} (a_y - 2)^2, \quad F(x,\eta) := \eta_x, \quad g(x,\eta) \equiv 0,
\end{align*}
where $b>0$ will be chosen later. Fix also $\A := [1,3]$. In our case, the Hamiltonian is:
\[
H(x,p):= \min_{a_x \in \A_{-x}^d} \Big\{\sum_{y\neq x} b(a_{xy} - 2)^2 + a_{xy} p_y \Big\}.
\] So when the minimum is attained in the interior of $\A^d_{-x}$, we have for all $y\neq x$ that:
\begin{align}\label{eqn:H_form}
\begin{split}
    &\gamma^*_y(x,p) = 2- \frac{p_y}{2b}, \quad H(x,p) = \sum_{y\neq x} \Big(2p_y - \frac{p_y^2}{4b}  \Big),
\end{split}
\end{align} and so:
\begin{align*}
    &\gamma_y^*(x,\Delta_x U(t_0,\cdot, \eta)) = 2 + \frac{U(t_0,x,\eta) - U(t_0,y,\eta)}{2b}, \\
    &H(x,\Delta_x U(t_0,\cdot, \eta)) = \sum_{y\neq x} \Big(2 [U(t_0,y,\eta) - U(t_0,x,\eta)] - \frac{[U(t_0,y,\eta) - U(t_0,x,\eta)]^2}{4b}  \Big) 
\end{align*}

Note that, should \eqref{eqn:H_form} hold, the Hamiltonian satisfies the regularity we assume at the very beginning of the paper in \eqref{eqn:H_regularity}. We will now explicitly verify the form of $H$. Along the way, we will show that this form for $H$ is not guaranteed for all problem parameter choices; however, we will prove that certain parameter choices will guarantee this form of $H$ and thus also the regularity required for the class of problems we consider. First, we will show why \eqref{eqn:H_form} holds when $|p| \leq 2b$.

The Hamiltonian has minimizer $a^*_x$ := $(a^*_{xy})_{y\in [d]}$:
\[
    a_{xy}^* = 
\begin{cases}
    3 & 2-\frac{p_y}{2b} \geq 3, \\
    2- \frac{p_y}{2b} & 2-\frac{p_y}{2b} \in \A, \\
    1 & 2-\frac{p_y}{2b} \leq 1.
\end{cases}
\]Note that: 
$2-\frac{p_y}{2b} \in \A$ occurs if and only if:
\begin{align}\label{eqn:H_nec}
  |p_y| \leq 2b.  
\end{align}

With this observation in mind, we are now interested in finding values of $b,T$ such that \[|[\Delta_x U(t_0,\cdot, \eta)]_y| \leq 2b,\] as this will lead to the desired form of the Hamiltonian in \eqref{eqn:H_form}.

The $p$-argument of $H$ will always be $\Delta_x U(t_0,\cdot,\eta)$ for some $(t_0,x,\eta) \in [0,T] \times [d]\times \calP([d])$ and so we are interested in a bound on this quantity. We will derive a bound for $|[\Delta_x U(t_0,\cdot, \eta)]_y|$ in terms of $T$. Then, by selecting $T$ and $b$ in concert with one another, we may guarantee \eqref{eqn:H_nec} and therefore confirm \eqref{eqn:H_form} holds for all inputs to $H$.

Let $\beta$ be the control that always chooses rate $2$. Using \eqref{eqn:me_cost_equivalence}, the fact that $\gamma^*$ is a minimizer, the choice of $f$, and the fact that $|F|\leq 1$, for $y\neq x$ we have:
\begin{align*}
    |[\Delta_x U(t_0,\cdot, \eta)]_y| 
    &= \Big|\EE_{(y,\eta)} \Big(\int_{t_0}^T f(\calX_t, \gamma^*(t,\calX_t)) + F(\calX_t, \calL(\calX_t)) dt\Big) \\
        &\qquad + \EE_{(x,\eta)}\Big(\int_{t_0}^T f(\calX_t, \gamma^*(t,\calX_t)) + F(\calX_t, \calL(\calX_t)) dt\Big)\Big|\\
    &\leq \Big|\EE_{(y,\eta)} \Big(\int_{t_0}^T f(\calX_t^\beta, \beta) + F(\calX_t^\beta, \calL(\calX_t^\beta)) dt\Big)\Big| \\
        &\qquad + \Big|\EE_{(x,\eta)}\Big(\int_{t_0}^T f(\calX_t^\beta, \beta) + F(\calX_t^\beta, \calL(\calX_t^\beta)) dt\Big)\Big|\\
    &\leq 2T.
\end{align*} So we can guarantee that $|[\Delta_x U(t_0,\cdot, \eta)]_y| \leq 2b$ by requiring that:
$
   T\leq b.
$ 
We thus choose $T=1/2$ and $b=4$ and this guarantees the smoothness of the Hamiltonian for any value of the master equation solution and also for any number of states $d$. While we could have chosen other values, these choices result in an easily interpretable picture as a consequence of the strong cost of rate selection and more prompt horizon. 

\subsection{Quadratic Cost---Low Dimensional Results}

Now we present numerical results for the low-dimensional cases $d=2, 3$ corresponding to this example. We begin with the case $d=2$; the following graphs are from the DGME.

\subsubsection{The master equation's solution for $d=2$, visualizations via DGME}

Here, we take an in-depth look at Figure \ref{fig:U_d_2}. The horizontal axis of the graph indicates $\mu_1(t)$, the fraction of $\mu$ in state $1$ at time $t$. The vertical axis is the value of the approximated master equation solution. The neural network is evaluated at a series of time steps and each time step is given a distinct  color, as indicated by the legend on the right.

\begin{figure}
    \centering
    \begin{tabular}{cc}
\includegraphics[width=75mm]{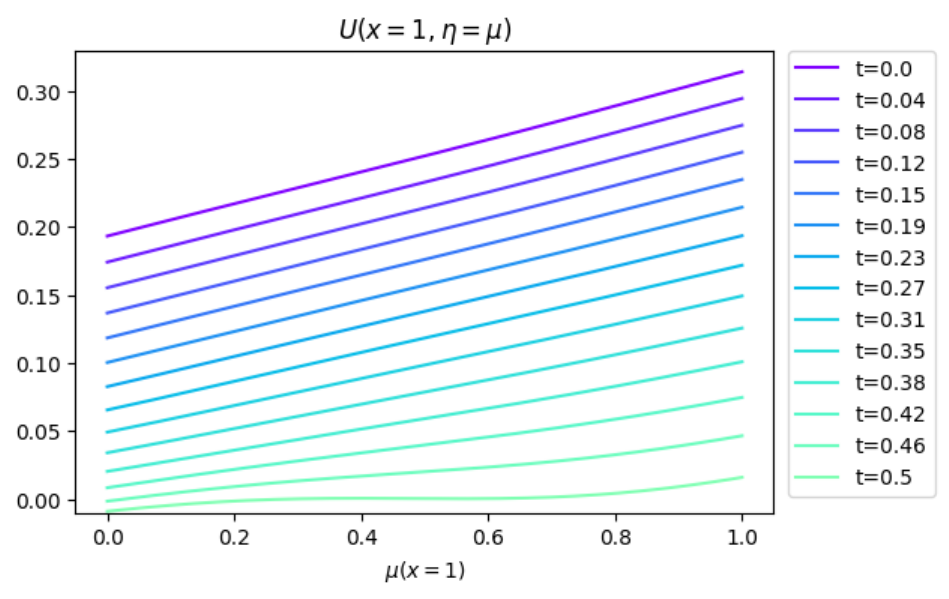}  & \includegraphics[width=75mm]{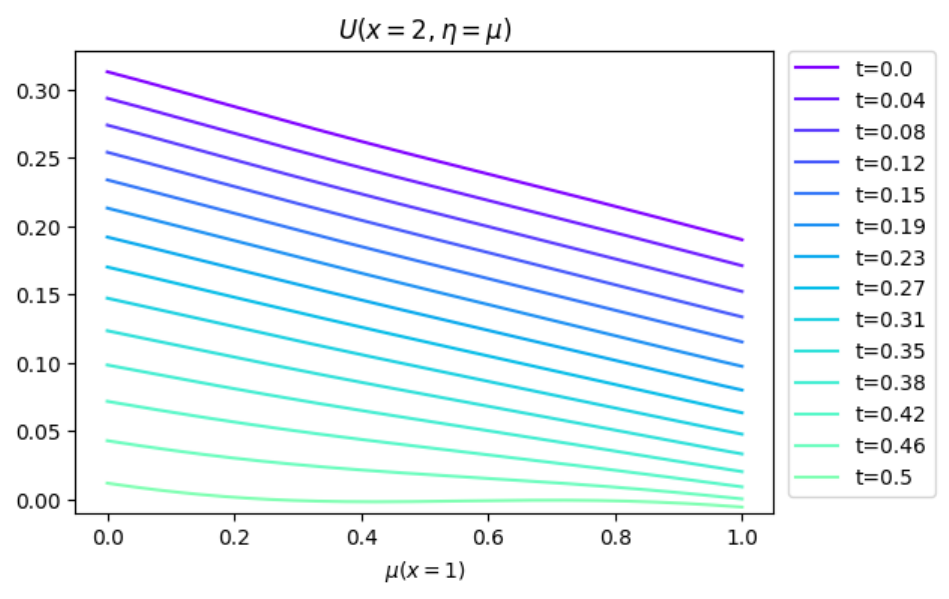} \\
\end{tabular}
\noindent
    \caption{The master equation's solution over time for $d=2$.}
    \label{fig:U_d_2}
\end{figure}

While there is no known explicit formula to sanity check our solution, we can make several qualitative judgements at the start. At the terminal time $T=1/2$, the graph indicates that the function is approximately zero; this agrees with the fact that the terminal condition is $g\equiv 0$ and $U(T,x,\eta) = g(x,\eta)= 0$. Second, the graph appears to respect the monotonicity of the cost. That is, for any fixed time $t_0\in [0,T]$ the graph satisfies: \begin{align*}
    U(t_0, 1, \eta) \leq U(t_0, 1, \eta+(1,-1)\eps),
\end{align*} for all $\eta$ in the interior of $\calP([d])$, and whenever $\eps \ge 0$ is small enough so that $U$ is defined. The reverse inequality is true for $x=2$. We expect this cost monotonicity to come from the mean field cost $F$ since a representative player experiences a lower cost when a smaller fraction from the distribution $\mu$ are in his or her own state. Finally, by \eqref{eqn:me_cost_equivalence} and definition of $J$, we expect $U(t,x,\eta)\to 0$ monotonically as $t\to T$ and indeed this is the case.

\subsubsection{Measuring DGME and DBME Losses}

In Figure \ref{fig:losses} we give loss graphs for the DGME and DBME algorithms. The DGME algorithm trains one neural network to the entire master equation and so its loss graph is straightforward: for every epoch, we average the losses over the thirty steps in that epoch and present the losses. The DBME however is the result of neural networks being trained on the output of a neural network at a future time step, and so the DBME returns $N$ networks. Moreover, recall that the final DBME ``neural network" is set to be the terminal cost and so the first neural network that is trained with an actual loss function is the penultimate network $\calU_{N-1}$. In practice, we found that training $\calU_{N-1}$ for more epochs than the other networks resulted in better results for networks closer to the initial time. Therefore, Figure \ref{fig:losses} also presents three graphs for the DBME trained with a constant-sized partition of $|\pi|=0.01$. The first DBME graph depicts the losses of $\calU_{49}$ by itself plotted on a log scale, while the second graph gives the losses for the networks $\calU_{0}, \calU_{10}, \calU_{20}$, $\calU_{30}$, and $\calU_{40}$.

\begin{figure}
    \centering
    \begin{tabular}{c}
\includegraphics[width=70mm]{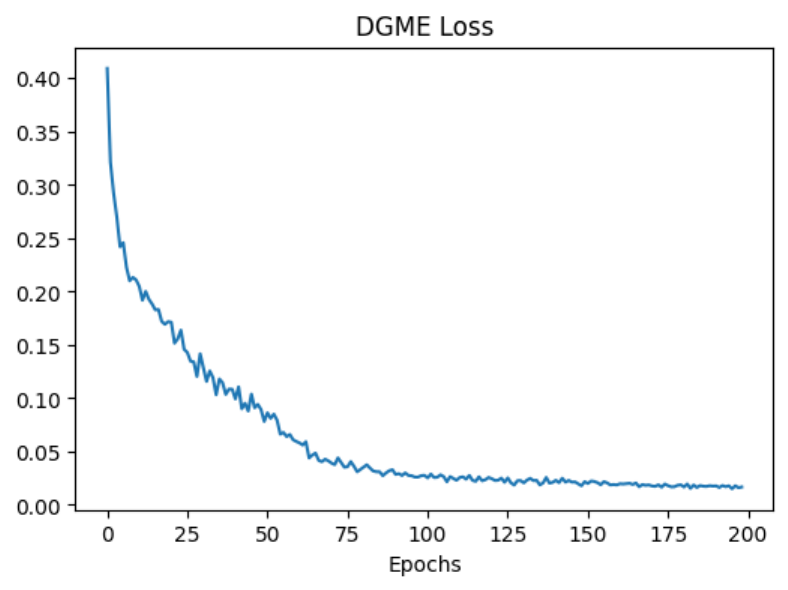} 
    \end{tabular}
    \begin{tabular}{cc}
\includegraphics[width=70mm]{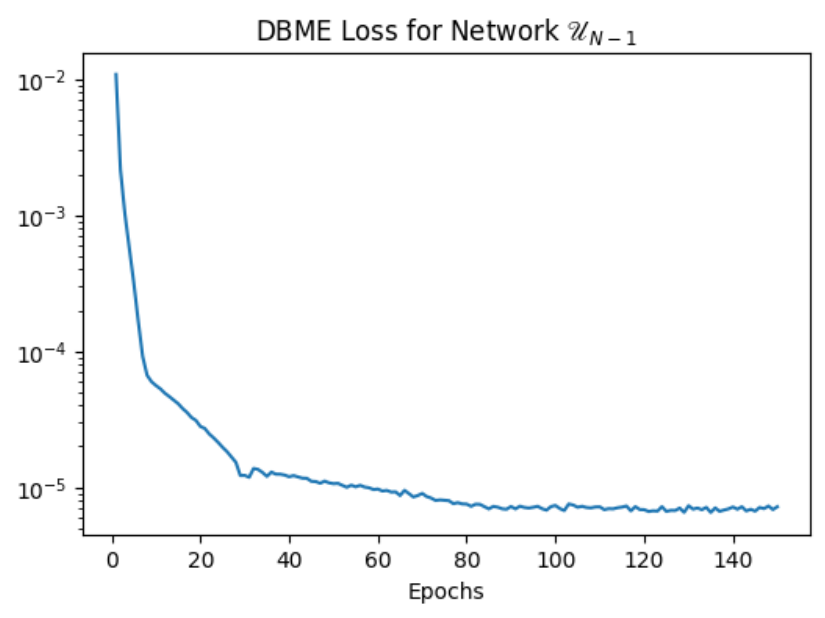} &  \includegraphics[width=67mm]{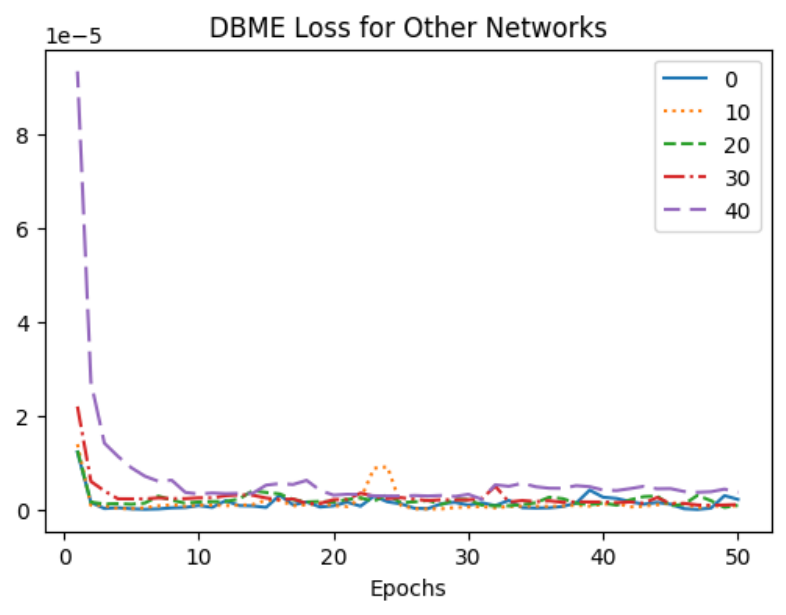}
\end{tabular}
\noindent
    \caption{At the top: DGME losses for $d=2$, averaged over thirty steps per epoch. At the bottom left, the average DBME losses for the penultimate network plotted on a log scale, and at the bottom right, the average losses for some DBME networks. }
    \label{fig:losses}
\end{figure}

\subsubsection{Agreement between the DGME and DBME algorithms}

The result from the DBME appears similar given the partition step size is small enough. In Figure \ref{fig:avg_diff} we graph the error between the prior DGME model and its DBME counterpart over all times in the partition $t_i \in \pi$.  The error is computed as the average difference over the state space $[d]\times\calP([d])$ and the horizontal axis is the difference on $[0,T]$.

\begin{figure}
    \centering
    \includegraphics[width=90mm]{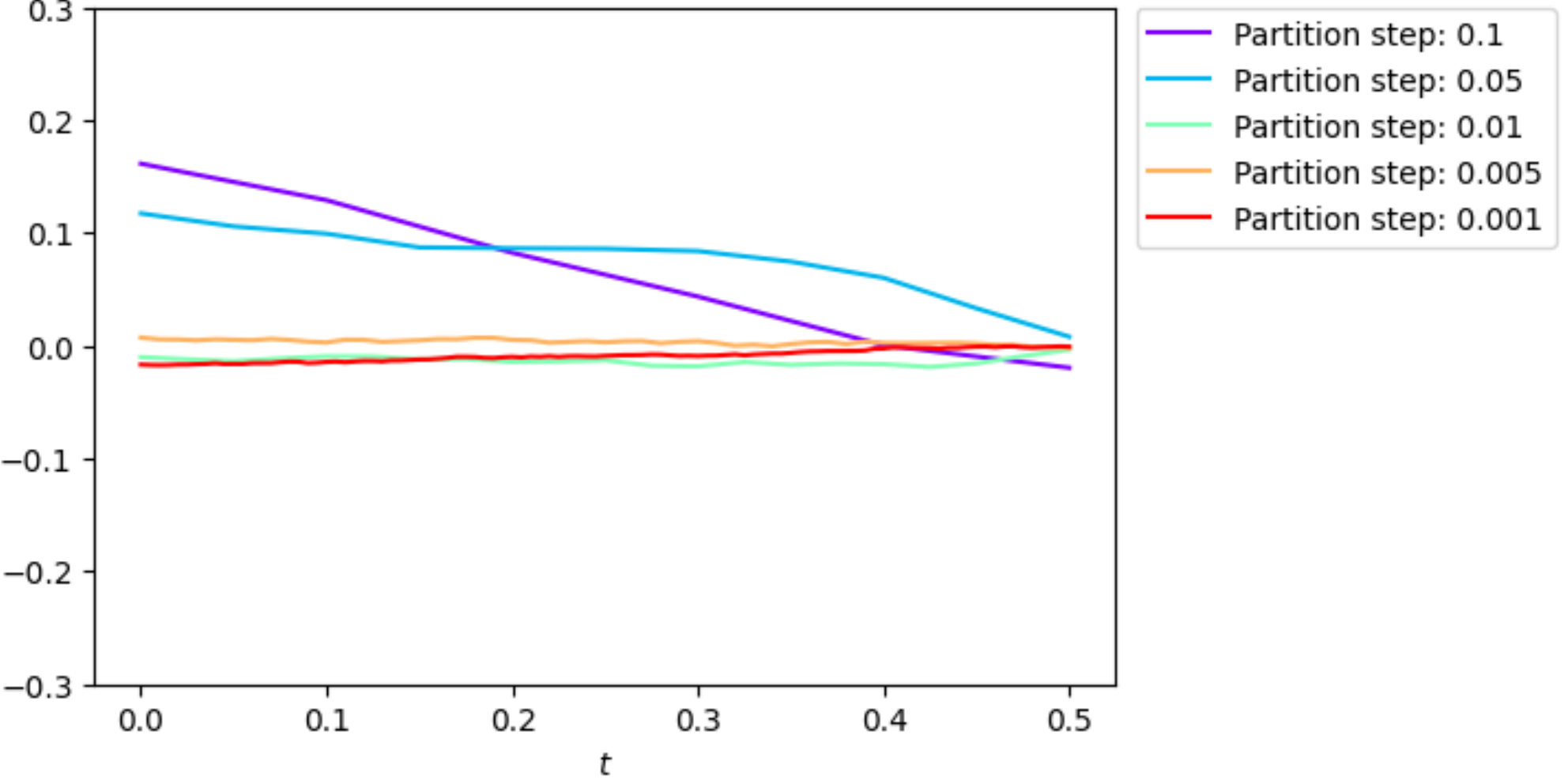}
    \caption{For $d=2$, the average difference between the DBME and DGME solutions as the partition step decreases.}
    \label{fig:avg_diff}
\end{figure}

In Figure \ref{fig:avg_diff}, for the partitions $\pi$ such that $|\pi| \leq 0.01$, the DGME and DBME algorithms agree on the value of $U$ on average. For larger values of $|\pi|$, there is some error that propagates backward and worsens the estimate at the initial time $t=0$.

\subsubsection{The master equation's solution for $d=3$, visualizations via DGME}

Here, we break down the result in Figure \ref{fig:U_d_3}. It is qualitatively similar to Figure \ref{fig:U_d_2} (for $d=2$) but we present it differently. In Figure \ref{fig:U_d_3}, each sub-plot corresponds to a different time. The simplex $\mathcal{P}([3])$ is identified with a triangle in dimension 2, where the x-axis corresponds to the density in state $2$ and the y-axis corresponds to the density in state $1$. The horizontal axis is $\mu_2(t)$ while the vertical axis is $\mu_1(t)$. The color intensity gives the value of the approximate master equation solution. For clarification, the bottom left corner is the point in $\calP([d])$ where $\mu_3(t) = 1$. In the first image of Figure \ref{fig:U_d_3} we notice the same cost monotonicity as in the two-dimensional case. Meanwhile, we note that the final image (where $t=0.5$) is approximately zero as desired.

\begin{figure}
\centering
\begin{tabular}{ccc}
\includegraphics[width=0.33\linewidth]{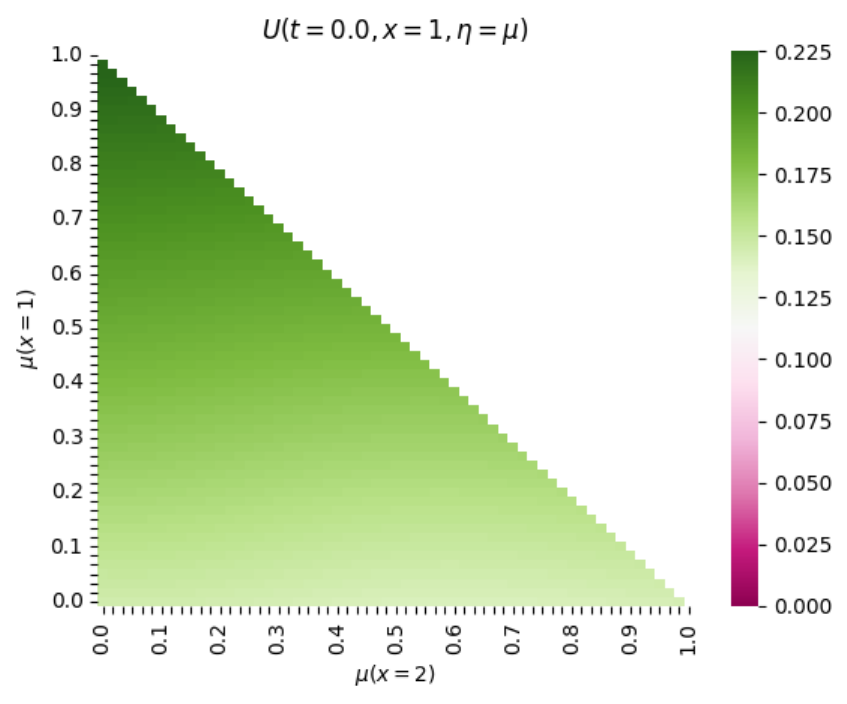}  & \includegraphics[width=0.33\linewidth]{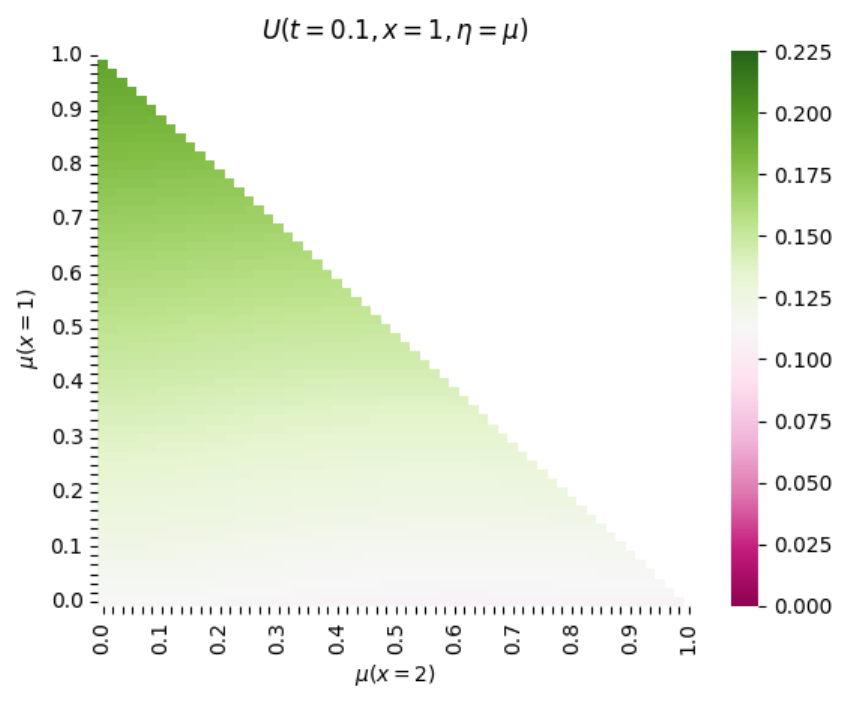} &
\includegraphics[width=0.33\linewidth]{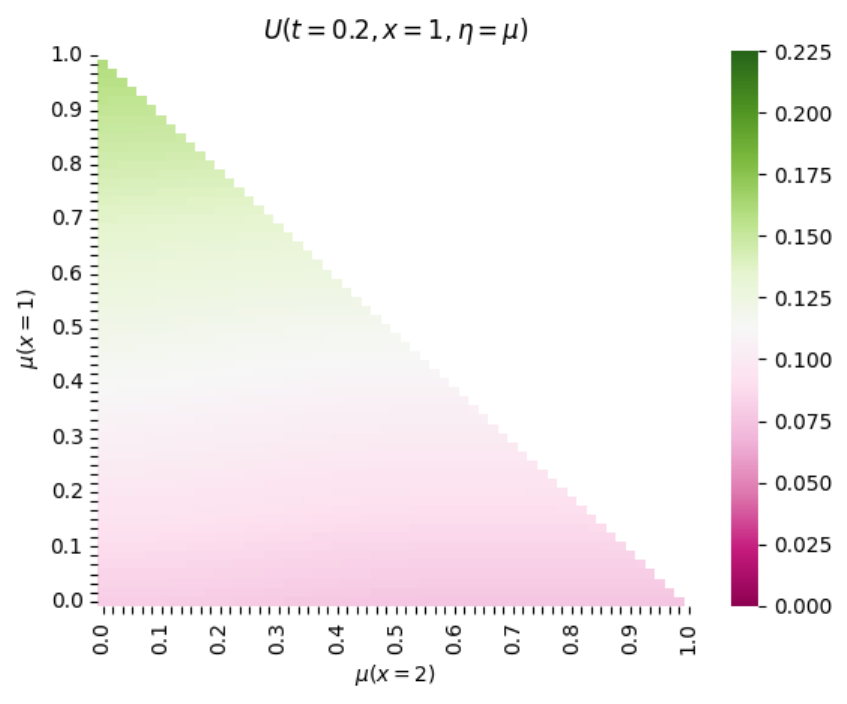} \\ 
\includegraphics[width=0.33\linewidth]{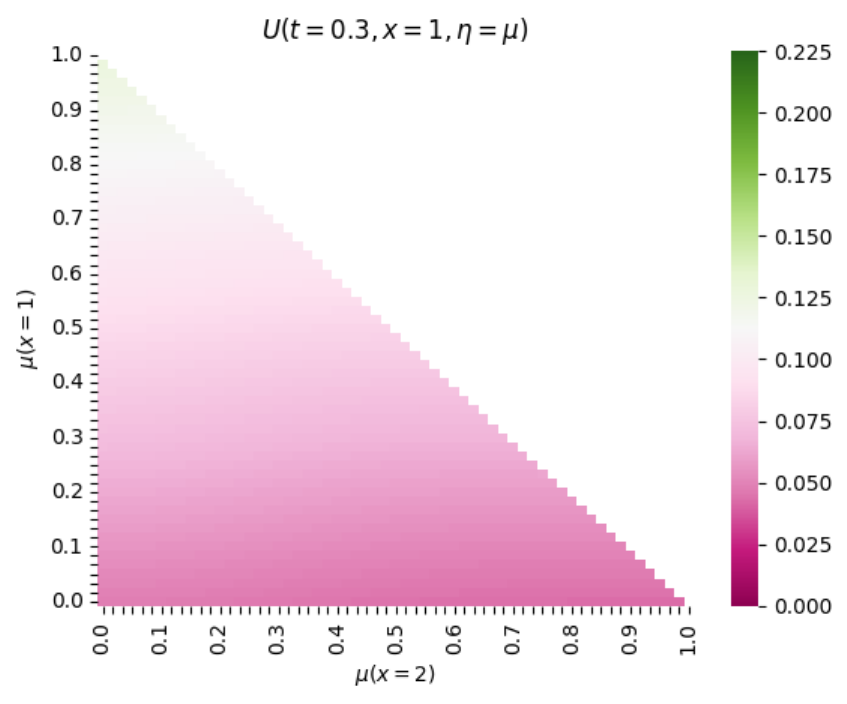} &
\includegraphics[width=0.33\linewidth]{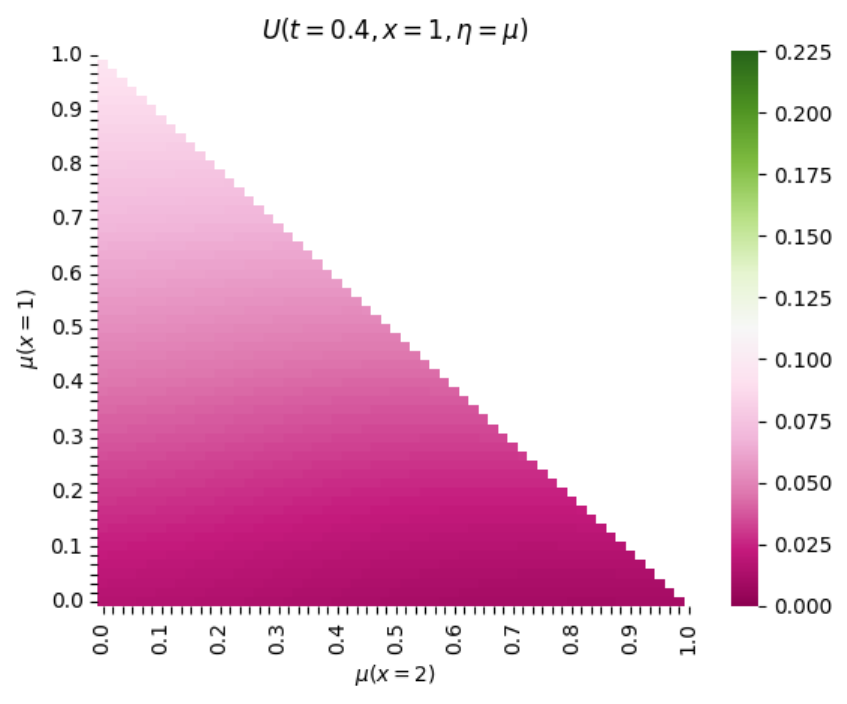} & 
\includegraphics[width=0.33\linewidth]{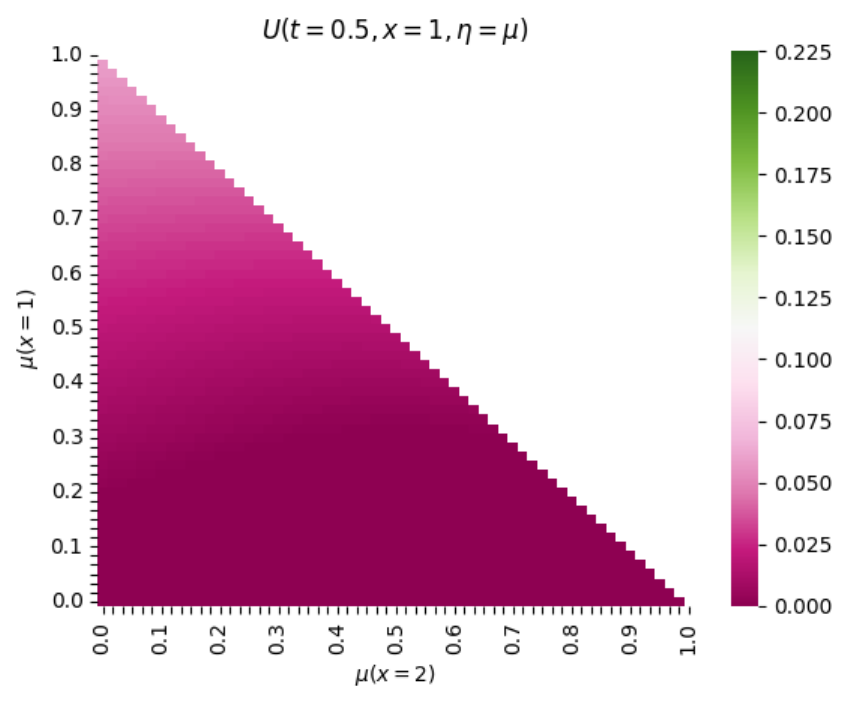} 
\end{tabular}
\caption{The master equation's solution on the simplex $\calP([3])$ over time.}
\label{fig:U_d_3}
\end{figure}

Once we have approximately learned the master equation solution, we may compute the corresponding control by plugging the solution into the rate selector $\gamma^*$. Then, we are able to plot the trajectory of the player population over time. In Figure \ref{fig:mu_traj}, we plot the trajectory as a proportion for a chosen initial state, and then on the simplex $\calP([3])$. On $\calP([3])$, the dark purple end of the line corresponds to the initial state $\mu(t=0)$, while the light yellow end corresponds to the terminal state $\mu(t=T)$. 

\begin{figure}
    \centering
    \begin{tabular}{c c}
\includegraphics[width=80mm]{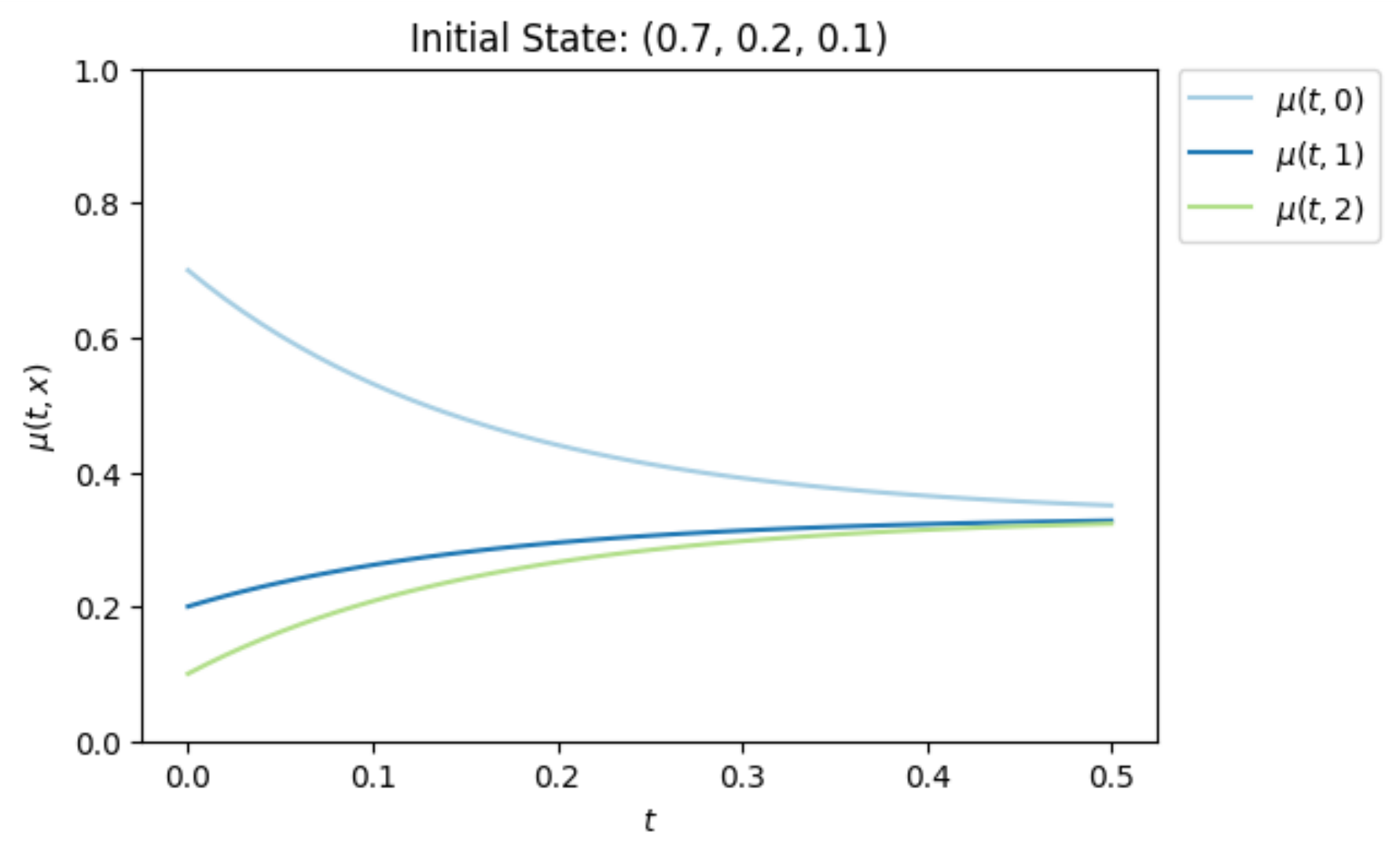}   &  \includegraphics[width=50mm]{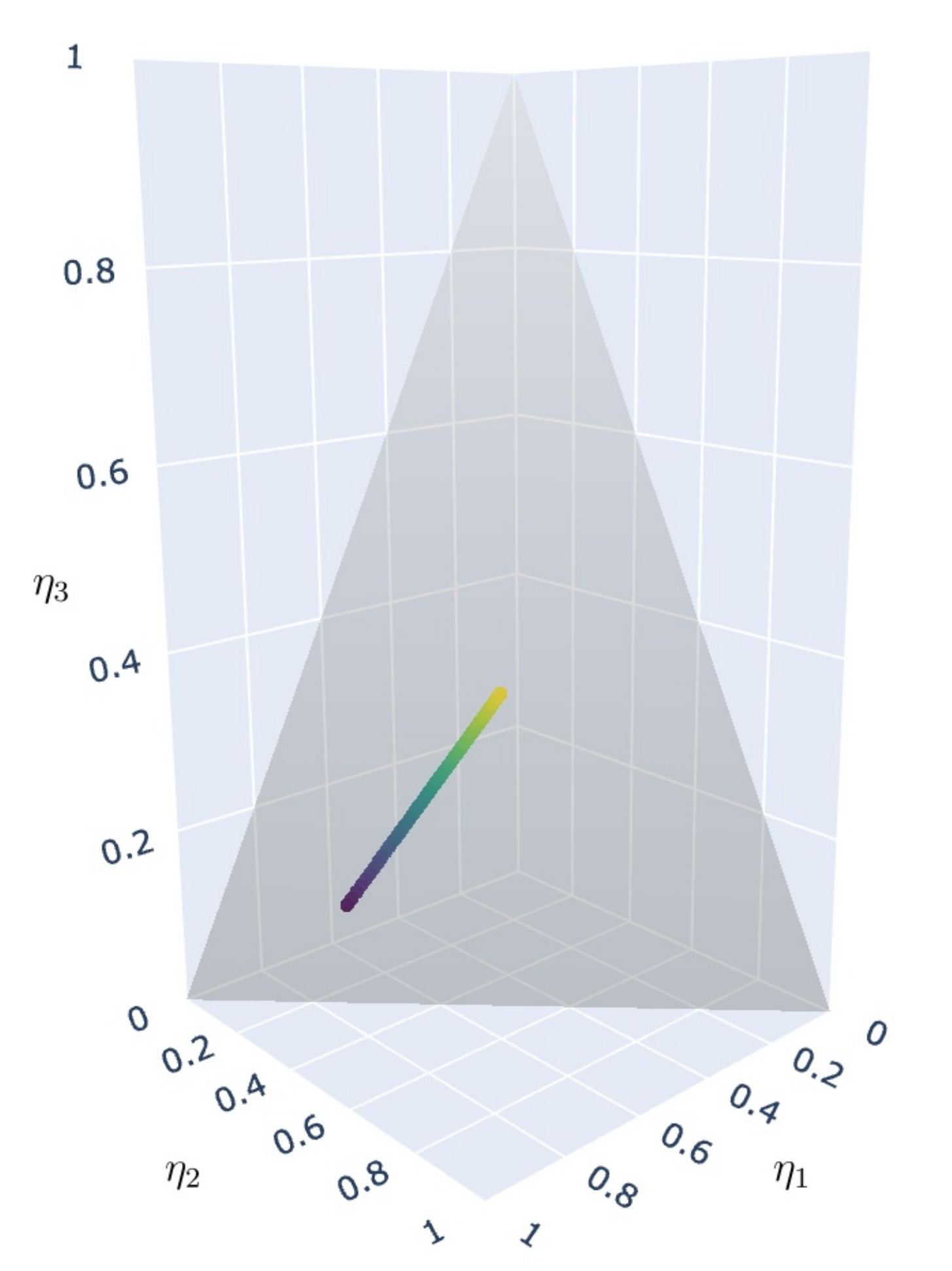} \\
\end{tabular}
    \caption{Two ways of visualizing the MFG equilibrium $\mu^{0, (.7,.2,.1)}$.}
    \label{fig:mu_traj}
\end{figure}

\subsection{Quadratic Cost---High Dimensional Results}

How the DGME and DBME deal with the curse of dimensionality is of particular interest. To study the run-time efficiency as dimension increases, we compute the relative increase in run-time for each dimension by dividing the corresponding running time by the running time required for the previous
dimension. 

\begin{figure}
    \centering
    \includegraphics[width=70mm]{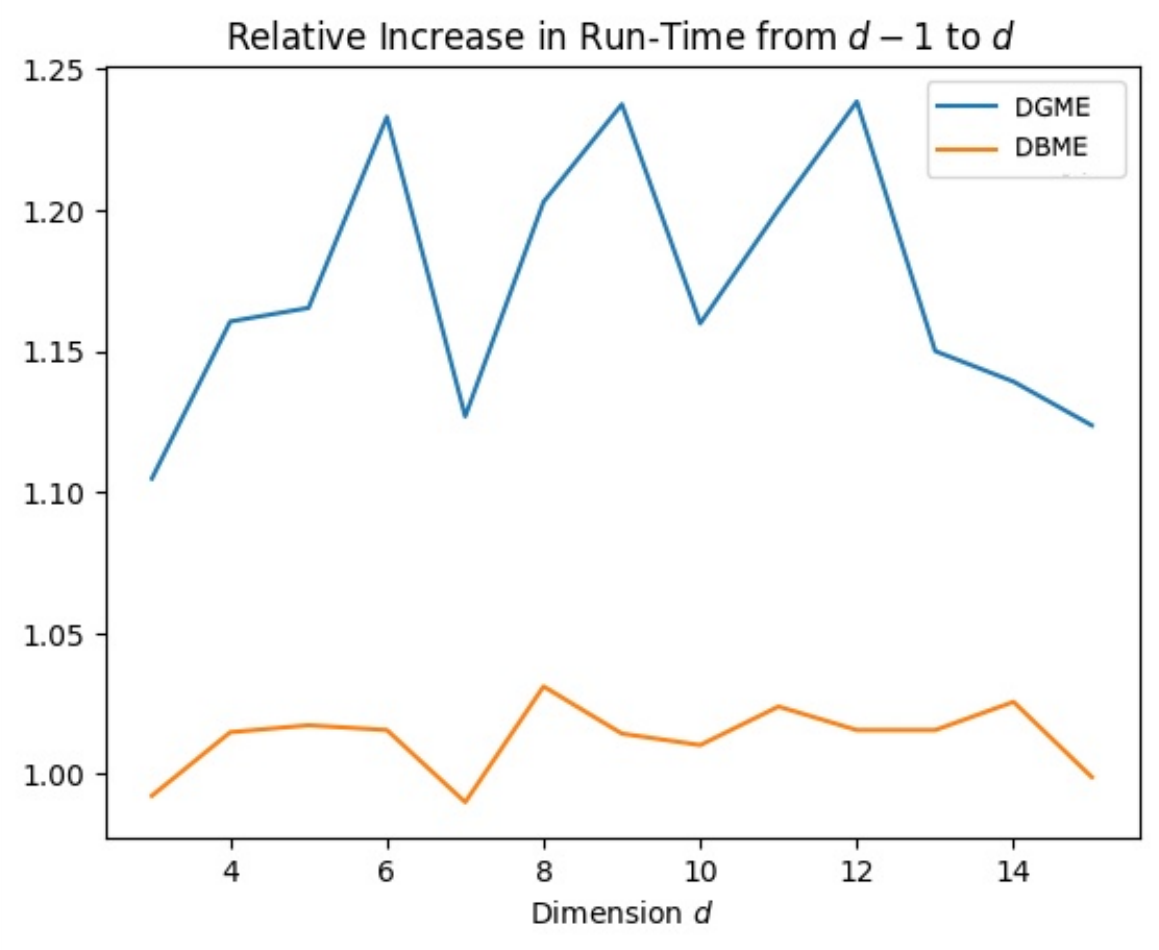}
    \caption{Comparing the increase in run-time of each algorithm as dimension increases, and when requiring the same loss.}
    \label{fig:run_time_increase}
\end{figure}

In Figure \ref{fig:run_time_increase}, we note that the DBME run-time increases at around the rate of increase in dimension; meanwhile, the DGME run-time increases faster with an equivalent increase in dimension. The lower relative increase for the DBME run-time indicates that it may be better suited to high-dimensional problems. Moreover, the rate of relative increase for the DBME agrees with other known deep backward schemes; see for example the deep backward scheme for an optimal switching problem studied in \cite[Figure~2]{2210.03045}. 

We also note that the size of our state space increases quadratically in $d$. Namely, we are considering $[d]\times \calP([d])$. The simplex $\calP([d])$ is $(d-1)$-dimensional and so the size of our state space is $d^2-d = \calO(d^2)$.

\subsection{Cybersecurity Example} 

In this section, we study another popular MFG model that represents the cybersecurity of a network of computers. The cybersecurity model was introduced by \cite{MR3575619}, revisited by \cite[Section~7.2.3]{CarmonaDelarue_book_I}, and computed numerically by \cite[Section~7.2]{MR4368188}. We recall the problem here, detail its master equation, and solve using the DGME and DBME methods. 

In the cybersecurity MFG model, a continuum of agents try to minimize their costs that come from two sources: being infected with a computer virus, and paying for computer security software. An agent's computer can be infected by another infected computer or by a hacker (considered exogenous to the model). Intuitively, the mean field interaction comes through the infection of a computer by other infected computers. Once infected, a computer can recover back to its original state. Agents that elect to defend their computers with the software are infected at a slower rate than undefended computers. 

So, every agent has two internal states: defended/undefended ($D,U$) and susceptible/infected ($S,I$). Hence there are four states ($d=4$):
$
    [d]:=\{DS, US, DI, UI\}.
$  
Every agent pays a cost while infected, $k_I>0$, and a cost while defending, $k_D>0$. The agent pays the most cost per second when both infected and defending. Specifically, the running cost $f$ is given by:
\begin{align*}
    f(x,a):= k_D \1_{\{DS, DI\}} (x) + k_I \1_{\{DI, UI\}}(x).
\end{align*} The mean field cost and terminal cost are both set to zero,
$
    F=g\equiv 0.
$  
The agent's only choice is to switch his or her defended or undefended status; this occurs with rate $\rho$, and the corresponding action set is $\A = \{0,1\}$. 

The rate a hacker hacks into any computer is determined by the parameter $v_H$ which for defended or undefended computers is augmented by $q^D_{H}$ and $q^U_H$, respectively. Correspondingly, the recovery parameters are $q^D_R$ and $q^U_R$. Infections from defended and undefended computers are augmented by the parameters $\beta_{DD}, \beta_{UU}, \beta_{UD}, \beta_{DU}>0$ in the following rate matrix. Namely, agents transition according to the rate matrix:
\begin{align*}
    (M(x,y; a))_{x,y\in [d]}:= \begin{bmatrix}
        \ast & \rho a & P^\eta_{DS, DI} & 0\\
        \rho a & \ast & 0 & P^\eta_{US, UI} \\
        q^D_R & 0 & \ast & \rho a \\
        0 &  q^U_R & \rho a &\ast
    \end{bmatrix},
\end{align*} where
\begin{align*}
    & P^\eta_{DS, DI} := v_H q^D_H + \beta_{DD} \eta(DI) + \beta_{UD} \eta(UI), \\
    & P^\eta_{US, UI} := v_H q^U_H + \beta_{UU} \eta(UI) + \beta_{DU} \eta(DI),
\end{align*} and where the asterisk denotes the negative of the sum of the elements along that row. Note that agents can only determine their rate matrix by choosing either $a=0$ or $1$ and all other parameters are fixed. That said, $P^\eta_{DS, DI}$ and $P^\eta_{US, UI}$ change based on the proportion of infected undefended and defended agents' computers, respectively; it is through these terms that the agents experience a mean field-type interaction. 

With this said, we can write down the pre-Hamiltonians, the functions inside the Hamiltonian's minimum, for each state:
\begin{align*}
    &\tilde H(DS, \Delta_{DS} U(t,\cdot, \eta); a) := k_D + P^\eta_{DS, DI} (U(t, DI, \eta) - U(t, DS, \eta)) \\
    &\qquad\qquad\qquad\qquad\qquad\qquad\qquad\qquad\qquad\qquad +\rho a (U(t,US, \eta) - U(t, DS, \eta)),\\
    &\tilde H(US, \Delta_{US} U(t,\cdot, \eta); a) := \rho a ( U(t, DS, \eta) - U(t, US, \eta)) + P^\eta_{US, UI} (U(t,UI, \eta) - U(t, US, \eta)) , \\
    &\tilde H(DI, \Delta_{DI} U(t,\cdot, \eta); a) := k_D + k_I + \rho a (U(t,UI, \eta) - U(t, DI, \eta)) \\
    &\qquad\qquad\qquad\qquad\qquad\qquad\qquad\qquad\qquad\qquad+ q^D_R (U(t,DS, \eta) - U(t,DI, \eta)), \\
    &\tilde H(UI, \Delta_{UI} U(t,\cdot, \eta) ; a) := k_I + \rho a (U(t,DI, \eta) - U(t, UI, \eta)) + q^U_R (U(t,US,\eta) - U(t,UI, \eta)).
\end{align*} Taking the minimum over $a\in \{0,1\}$ yields the usual Hamiltonian. We can then write the optimal rate matrix:
\begin{align*}
    &(\gamma^*_y(x, \Delta_x U(t,\cdot, \eta)))_{x,y\in [d]} :=\\ &\quad\begin{bmatrix}
        \ast & \rho \1_{\{U(t,US,\eta) < U(t,DS, \eta)\}} &P^\eta_{DS, DI}  & 0\\
        \rho \1_{\{U(t,DS,\eta) < U(t,US, \eta)\}} & \ast & 0 & P^\eta_{US, UI} \\
        q^D_R & 0 & \ast & \rho \1_{\{U(t,UI,\eta) < U(t,DI, \eta)\}} \\
        0  & q^U_R & \rho \1_{\{U(t,DI,\eta) < U(t,UI, \eta)\}} & \ast
    \end{bmatrix}.
\end{align*} 
So, we can write the master equation as: for all $t\in [0,T]$, $x\in [d]$, and $\eta \in\calP([d])$,
\begin{align*}
    &\partial_t U(t,x,\eta) + \sum_{y,z \in [d]} \eta_y D^\eta_{yz} U(t,x,\eta) \gamma^*_z(y, \Delta_y U(t,\cdot, \eta)) + H(x,\Delta_x U(t,\cdot, \eta)) = 0, \\
    &U(T,x,\eta) =0.
\end{align*} 

By running either algorithm then, we can reconstruct the optimal control and hence the trajectory of the measure. Recall the MFG system \eqref{eqn:mfg} and its solution $(u,\mu)$ (removing the superscripts). We demonstrate the measure trajectory $\mu$ in the first row of Figure \ref{fig:two_measures}, using the solution computed by the DGME. The corresponding cost $u$ is displayed in the second row of Figure \ref{fig:two_measures}. In Figure \ref{fig:two_measures}, the legend depicts the four states $DS$, $US$, $DI$, and $UI$; the prefix of DGM shows the DGME's result against a finite difference method solver. 
Further results for other initial distributions are provided in Appendix~\ref{app:more-plots-cybersecurity}.

\begin{figure}
    \centering
    \begin{tabular}{c}
    \includegraphics[width = 90mm]{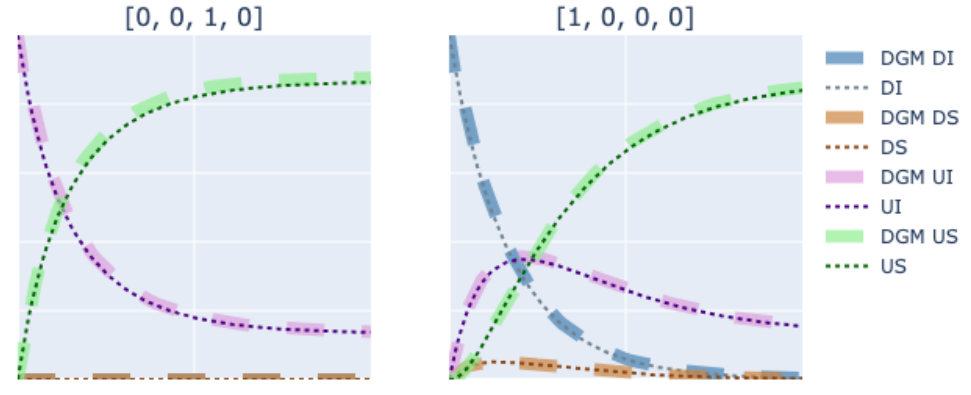}  \\
    \includegraphics[width = 90mm]{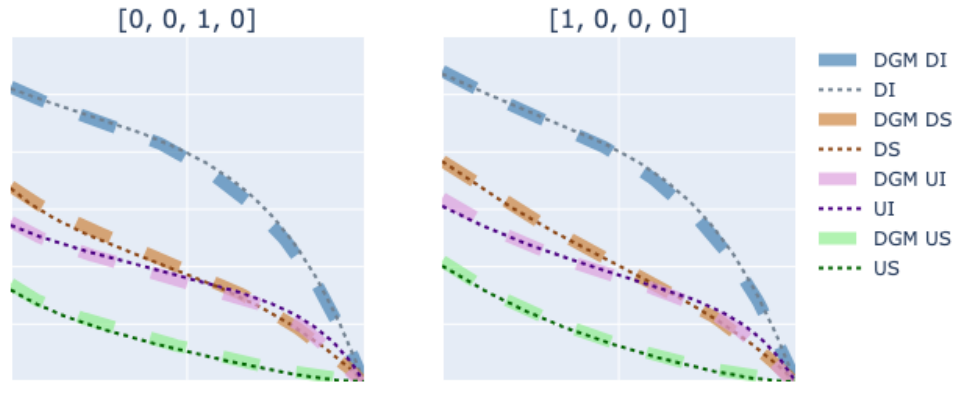}
    \end{tabular}
    
    \caption{The top row gives two examples of $\mu$ given different initial conditions. The bottom row gives the corresponding $u$ for these measures.}
    \label{fig:two_measures}
\end{figure}

\section{Conclusion}

Motivated by the computation of Nash equilibria in games with many players, we proposed two algorithms, the DBME and the DGME, that can numerically solve the master equation of finite-state MFGs. For each method, we proved two results: the loss can be made arbitrarily small by some neural network, and conversely, if a neural network makes the loss small, then this neural network is close to the real solution of the master equation. 
Besides the theoretical analysis, we provided two in-depth numerical examples applying the methods and compare them at different dimensions. We found that the relative increase in computation time is superior for the DBME. Our numerical tests show that the proposed methods allow us to learn the master equation solution for any initial distribution, in contrast with traditional forward-backward methods, which work only for a single fixed initial distribution. Hence our methods open new potential applications of MFGs, when the initial distribution is not known in advance or when the population distribution may deviate from its expected trajectory.

Several directions are left for future work. First, our methods have the same limitations as the DBDP and the DGM, and will benefit from progress made on these methods. For instance, performance might be improved with a better understanding of how to choose the neural network architecture depending on the PDE under consideration. It should also be possible to improve the training speed by a suitable choice of sampling method. In particular, we leave for future work the incorporation of active sampling to the proposed methods. 

On the theoretical side, we have mostly focused on the approximation error. It would be interesting to obtain an error rate in terms of the number of neurons and the dimension, that is, the number of states of the MFG. At this stage, numerical evidence suggests that our approaches work well even with a large number of states. However, as for many other methods, such as the DGM and the DBDP in particular, it is not clear how to prove rigorously that neural networks break the curse of dimensionality. Furthermore, it would also be interesting to study other error sources, whether from sampling regimes or gradient descent convergence. Ultimately, the goal would be to comprehensively understand the generalization error of neural network-based methods for master equations.

Last, real-world applications are beyond the scope of the present paper. This would involve studying in detail the link between the master equation in MFGs and systems of Bellman PDEs in finite-player games. Such systems can also be solved by deep learning. When the number of players is small, it is doable and more accurate to solve directly the PDE system. But as the number of players grows, it becomes easier to solve the master equation, and the MFG provides a good approximation of the finite-player game equilibrium. However, the precise trade-off depends on the model under consideration, which is left for future work in a case-by-case study.

\vspace{5pt}
{\bf Acknowledgement.} We thank the anonymous AE and the referee for their insightful comments, which helped us improve our paper.

\section*{Appendix A. Proof of Lemma~\ref{lemma:sample_conv}}

    Let $|\cdot|_1$ denote the $L^1$-norm in $\R^d$ and define $A:= |\calK|$. Choose:
    \[
    \bar\kappa \in \argmax_{\kappa\in\calP([d])}\tilde L_i(\kappa;\theta^i),
    \] so that:
    \[
    \tilde L_i(\bar\kappa ; \theta^i) = L_i(\theta^i).
    \]
    Since $\calK\subsetneq \calP([d])$,
    \begin{align*}
        0\leq  \tilde L_i(\bar\kappa;\theta^i) - \tilde L_i(\hat\kappa ; \theta^i). 
    \end{align*} Pick $\kappa_0 \in\calK$ such that:
    \[
    \kappa_0 \in \argmin_{\kappa \in \calK} |\kappa - \bar\kappa|.
    \] Then,
    \begin{align*}
        0&\leq  \tilde L_i(\bar\kappa;\theta^i) - \tilde L_i(\hat\kappa ; \theta^i) 
        \leq  \tilde L_i(\bar\kappa;\theta^i) - \tilde L_i(\kappa_0 ; \theta^i) .
    \end{align*} Using the above display, the definition of $\bar \kappa$, and the assumption that the neural networks are uniformly Lipschitz, as well as the Lipschitz continuity of $\bar H$, there exists $C>0$ depending on these Lipschitz constants such that: 
    \begin{align}\label{eqn:loss_conv_1}
        |L_i(\theta^i) - \tilde L_i(\hat \kappa;\theta^i)|&\leq C |\bar \kappa - \kappa_0|.
    \end{align} So it suffices to bound the right hand side of the above display in expectation. Let $p:\R^d\to\R^{d-1}$ be the projection onto the first $(d-1)$-coordinates of a vector in $\R^d$. Using that $\bar\kappa$, $\kappa_0 \in\calP([d])$ and the Cauchy--Schwartz inequality:
    \begin{align*}
        |\bar \kappa - \kappa_0|^2 &= \sum_{x\in [d]} (\bar\kappa_x - \kappa_{0,x})^2 \\
        &= \sum_{x = 1}^{d-1} (\bar\kappa_x - \kappa_{0,x})^2 + \Big(1 - \sum_{x=1}^{d-1} \bar\kappa_x - 1 + \sum_{x=1}^{d-1} \kappa_{0,x} \Big)^2 \\
        &= \sum_{x = 1}^{d-1} (\bar\kappa_x - \kappa_{0,x})^2 + \Big(\sum_{x=1}^{d-1} (\kappa_{0,x} - \bar\kappa_x) \Big)^2 \\
        &\leq \sum_{x = 1}^{d-1} (\bar\kappa_x - \kappa_{0,x})^2  + (d-1) \sum_{x = 1}^{d-1} (\bar\kappa_x - \kappa_{0,x})^2 \\
        &= d \sum_{x = 1}^{d-1} (\bar\kappa_x - \kappa_{0,x})^2 \\
        &= d |p(\bar\kappa - \kappa_0)|^2. 
    \end{align*} Plugging this result into \eqref{eqn:loss_conv_1},
    \begin{align*}
        |L_i(\theta^i) - \tilde L_i(\hat \kappa;\theta^i)|&\leq C\sqrt{d} |p(\bar \kappa - \kappa_0)|.
    \end{align*} Recalling that $\kappa_0$ is a minimum, noting that the difference is bounded above by the case when $p(\bar\kappa) = 0$, then since $\kappa_x \leq 1$ for all $x\in [d]$, we have:
    \begin{align*}
        \EE |p(\bar\kappa - \kappa_0)| &= \EE\Big[ \min_{\kappa \in\calK} |p(\bar\kappa - \kappa)|\Big] \leq \EE\Big[ \min_{\kappa \in\calK} |p(\kappa)|\Big] \leq \EE\Big[ \min_{\kappa \in\calK} |p(\kappa)|_1\Big]. \\ 
    \end{align*} Combining the two previous displays,
    \begin{align}\label{eqn:loss_conv_2}
        \EE|L_i(\theta^i) - \tilde L(\hat\kappa ; \theta^i)| \leq C\sqrt{d} \EE\Big[ \min_{\kappa \in\calK} |p(\kappa)|_1\Big].
    \end{align} Recall that sampling $\kappa$ uniformly on the simplex $\calP([d])$ is equivalent to creating $d$ exponentially distributed random variables $U_1, \dots, U_d$ (with rate $1$), and then setting each component of $\kappa:= (\kappa_x)_{x\in [d]}$:
    \begin{align*}
        \kappa_x := \frac{U_x}{\sum_{y\in [d]} U_y}.
    \end{align*} Therefore, 
    \begin{align*}
        |p(\kappa)|_1 = \sum_{x=1}^{d-1} \kappa_x \sim \beta(d-1 , 1),
    \end{align*} where by the last part of the expression we mean that $p(\kappa)$ has beta distribution with parameters $d-1$ and $1$. So in terms of the incomplete beta function $B(m;d-1,1)$ and the beta function $B(d-1,1)$, we can state the cumulative distribution function as:
    \begin{align*}
    \begin{split}
        \PP(|p(\kappa)|_1 \leq m) = \frac{B(m ;d-1,1)}{B(d-1,1)} = \frac{m^{d-1}}{d-1}.  
    \end{split}
    \end{align*} Using this fact, an integral identity, the fact that the $\kappa$'s are sampled i.i.d.:
    \begin{align}\label{eqn:loss_conv_3}
        \EE\Big[ \min_{\kappa \in\calK} |p(\kappa)|_1\Big] &= \int_0^1 \PP (|p(\kappa)|_1 \geq m)^A dm = \int_0^1 \Big(1 - \frac{m^{d-1}}{d-1}\Big)^A dm. 
    \end{align}


    Performing the substitution $\ell := m^{d-1} / (d-1)$, we get:
    \begin{align*}
        \int_0^1 \Big(1 - \frac{m^{d-1}}{d-1}\Big)^A dm&= (d-1)^{1/(d-1) - 1}\int_0^{1/(d-1)} (1 - \ell )^A \ell^{\tfrac{1}{d-1} - 1} d\ell.
    \end{align*}
    Notice that the integrand on the right hand side is the unnormalized probability density function of a beta distribution with parameters $1/(d-1)$ and $A+1$. Therefore the previous display is bounded above by extending the integral, then using Stirling's approximation:
    \begin{align}
    \begin{split}\label{eqn:integral_bound}
        (d-1)^{1/(d-1) - 1} \int_0^{1} (1 - \ell )^A \ell^{\tfrac{1}{d-1} - 1} d\ell &= (d-1)^{1/(d-1) - 1} \frac{\Gamma((d-1)^{-1}) \Gamma(A+1)}{\Gamma(A+1+(d-1)^{-1})} \\ 
        &\approx (d-1)^{1/(d-1) - 1}\Gamma((d-1)^{-1}) \sqrt{\frac{A+1}{A+1+(d-1)^{-1}}} \\
        &\qquad \times \Big(\frac{e}{A+1+(d-1)^{-1}}\Big)^{\tfrac{1}{d-1}} \Big(\frac{A+1}{A+1+(d-1)^{-1}}\Big)^{A+1} \\
        &\leq \tilde C_d \Big(\frac{e}{A+1+(d-1)^{-1}}\Big)^{\tfrac{1}{d-1}} \\
        &= \tilde C_d \calO\Big(\frac{1}{A^{1/(d-1)}}\Big),
    \end{split}
    \end{align} 
    where:
    \[
    \tilde C_d := (d-1)^{1/(d-1) - 1} \Gamma((d-1)^{-1}),
    \] and where we used the facts that:
    \[
    \Big(\frac{A+1}{A+1+(d-1)^{-1}}\Big)^{A+1} \leq 1 \quad \text{ and } \quad \sqrt{\frac{A+1}{A+1+(d-1)^{-1}}} \leq 1.
    \] So plugging in \eqref{eqn:loss_conv_3} into \eqref{eqn:loss_conv_2}, using \eqref{eqn:integral_bound}, and setting $C_d := C\tilde C_d\sqrt{d}$, we get:
    \[
    \EE|L_i(\theta^i) - \tilde L(\hat\kappa; \theta^i)| \leq \frac{C_d }{A^{1/(d-1)}}.
    \]

    
\qed

\section*{Appendix B. Additional Results for the Cybersecurity Example} \label{app:more-plots-cybersecurity}

In Figures \ref{fig:all_measures} and \ref{fig:all_costs}, we provide the plots for the cybersecurity example, but for many more initial conditions than offered earlier. The same labeling convention as in Figure~\ref{fig:two_measures} is adopted: the legend depicts the four states $DS$, $US$, $DI$, and $UI$; the prefix of DGM shows the DGME's result against a finite difference method solver.   In each figure, we provide a plot for every initial condition $\eta \in \calP([4])$ such that $\eta_x \in \{0,1/4,1/2,3/4,1\}$ for all $x\in [4]$. Figure \ref{fig:all_measures} displays the equilibrium trajectory $\mu^\eta(t)$ according according to the $\eta$ specified while Figure \ref{fig:all_costs} displays the cost along that equilibrium, $U(t,\cdot,\mu^\eta (t))$. The shape of the cost shows that the costliest situation for every initial condition and every time occurs when a player is both defended and infected. The cheapest situation is undefended and susceptible. For all initial conditions, there is a point in time when it is cheaper to be defended and susceptible than undefended and infected. Intuitively, since being infected carries a steep cost, one may want to pay for the computer's cybersecurity software to have a lower rate of incoming infection; if infected, transitioning to susceptible is more costly toward the end of the scenario because paying for the security software is costly and, without paying for the software, the transition to susceptible is slow. That said, the difference in total cost between defended susceptible and undefended infected in the latter part of the game is small, and all costs eventually converge to zero.

\begin{figure}
    \centering
    \includegraphics[width=165mm]{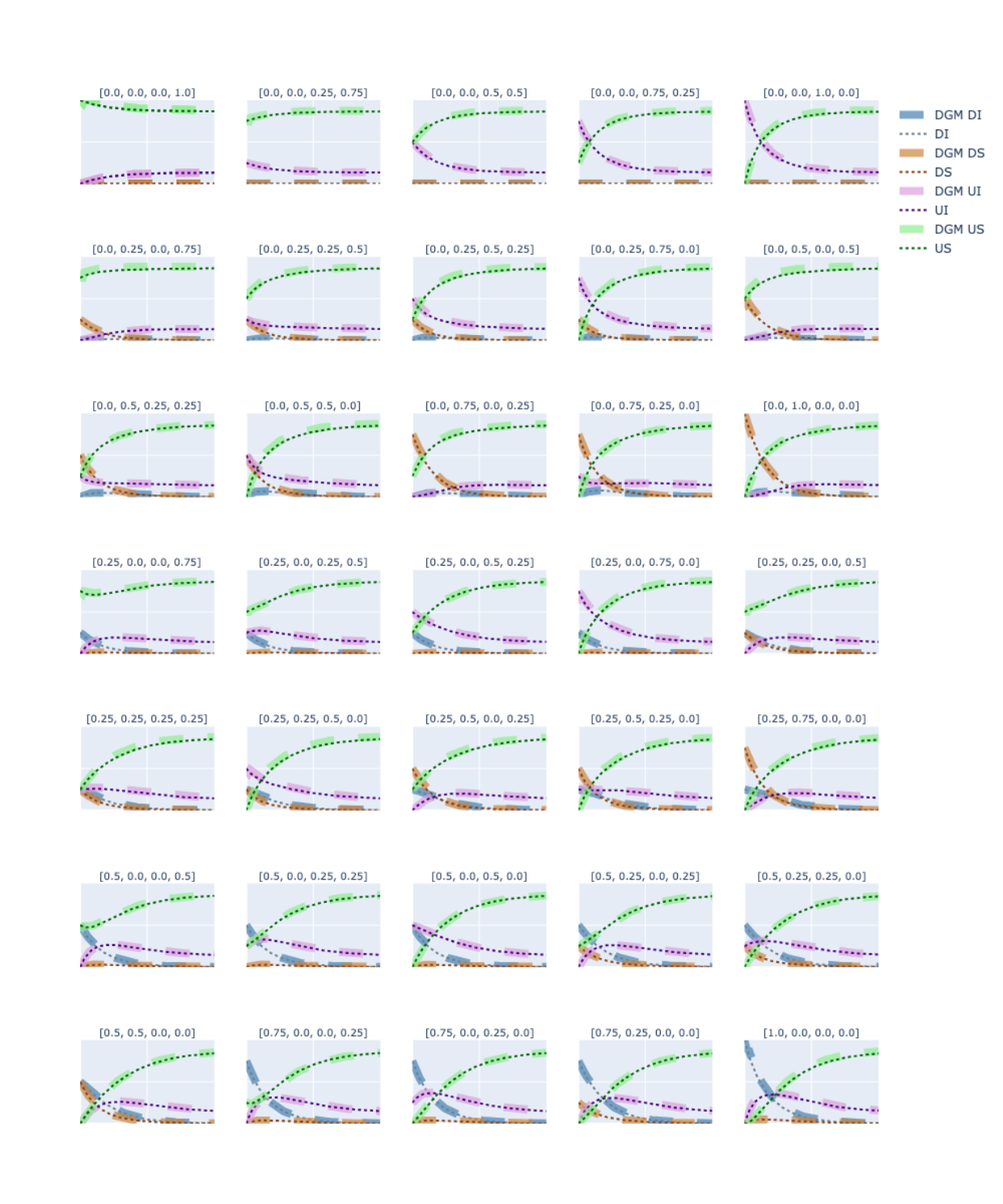}
    \caption{MFG equilibria $\mu$ for various initial conditions}
    \label{fig:all_measures}
\end{figure}

\newpage
\begin{figure}
    \centering
    \includegraphics[width=165mm]{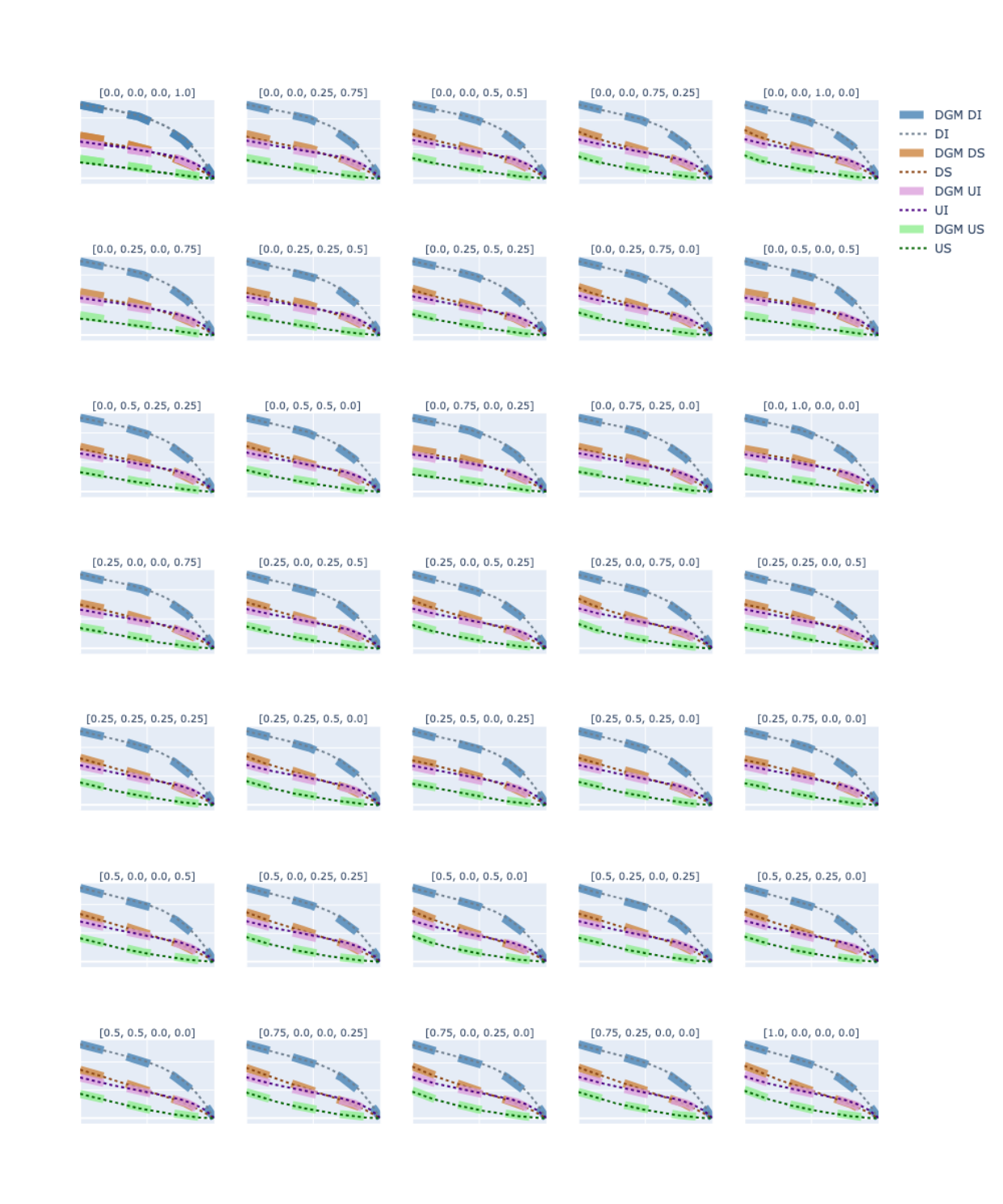}
    \caption{Cost trajectories $u$ for various initial conditions}
    \label{fig:all_costs}
\end{figure}


 \newpage
 \footnotesize
 \bibliographystyle{abbrv}
 \bibliography{trimmed_bib}



\end{document}